\documentclass[11pt,letterpaper,reqno]{amsart}
\usepackage{amsmath,amsthm,amsfonts,amssymb,mathtools,algorithm,algpseudocode}
\usepackage{comment,bookmark,bm,color,xcolor,tikz,hyperref,url,graphicx,subcaption,multicol}
\usepackage{array,booktabs}

\hypersetup{pdfstartview={FitH}}
\def\restrict#1{\raise-.5ex\hbox{\ensuremath|}_{#1}}
\newtheorem{Theorem}{{\bf Theorem}}[section]

\newtheorem{Corollary}[Theorem]{{\bf Corollary}}

\newtheorem{Proposition}[Theorem]{{\bf Proposition}}

\newtheorem{Lemma}[Theorem]{{\bf Lemma}}
\newtheorem{Remark}[Theorem]{{\bf Remark}}

\numberwithin{equation}{section}

\newcommand{\C}{\mathbb{C}}
\newcommand{\off}{\textup{off}}

\begin{document}

\title[The ``naive'' Jacobi eigenvalue algorithm]{Convergence and mixed-precision preconditioning for the ``naive'' Jacobi eigenvalue algorithm}
\author{Erna Begovi\'{c}~Kova\v{c}}\thanks{\textsc{Erna Begovi\'{c} Kova\v{c}}, University of Zagreb Faculty of Chemical Engineering and Technology, Maruli\'{c}ev trg 19, 10000 Zagreb, Croatia. \texttt{ebegovic@fkit.unizg.hr}}
\author{Marija Milolo\v{z}a~Pandur}\thanks{\textsc{Marija Milolo\v{z}a~Pandur}, University of Osijek School of Applied Mathematics and Informatics, Trg Ljudevita Gaja 6, 31000 Osijek, Croatia. \texttt{mmiloloz@mathos.hr}}
\author{Ana Perkovi\'{c}}\thanks{\textsc{Ana Perkovi\'{c}}, University of Zagreb Faculty of Chemical Engineering and Technology, Maruli\'{c}ev trg 19, 10000 Zagreb, Croatia. \texttt{aperkov@fkit.unizg.hr}}

\date{\today}

\renewcommand{\subjclassname}{\textup{2020} Mathematics Subject Classification}

\begin{abstract}
The paper studies a Jacobi-type method for the eigenvalue problem of general complex matrices with simple eigenvalues. The method applies elementary triangular similarity transformations in order to annihilate selected off-diagonal elements and, when convergent, produces highly accurate eigenvalues. We give a new proof of its asymptotic quadratic convergence and derive an explicit, verifiable bound that describes the region in which this convergence is guaranteed. To make the method applicable well beyond matrices already close to the diagonal form, we introduce a preconditioning strategy. We use two types of preconditioners, both based on theoretical convergence results. The preconditioner is computed at lower precision to reduce computational cost, the associated similarity transformation is applied either at working or at higher precision, to preserve spectral information, while the main algorithm performs at working precision. Numerical experiments demonstrate that the resulting algorithm is robust and produces very accurate eigenvalues.
\end{abstract}

\maketitle

\section{Introduction}\label{sec:intro}

The eigenvalue problem is one of the central problems in numerical linear algebra. In this paper we study an eigenvalue algorithm for general complex matrices with simple eigenvalues that originates in the PhD thesis~\cite{Zach}, where it was introduced as the \emph{``naive'' Jacobi algorithm}. 
The advantages of the method were promoted by Veseli\'c~\cite{VeselicNJ}. However, to the best of our knowledge, it has not been further studied. We revisit this method from a modern perspective. We prove a quantitative result for asymptotic convergence, prove the convergence on almost triangular matrices, and develop a preconditioning strategy that makes the method broadly applicable, exploiting mixed-precision arithmetic to reduce the cost of the preconditioning phase.

The well-known Jacobi method for symmetric matrices is a go-to choice for solving the eigenvalue problem on dense matrices. It is known for its global convergence properties~\cite{ShroffSchreiber89,Masc95,Drmac09,Hari15,BKHari17}, high relative accuracy~\cite{DeVe92,Mathias95,DrVe07-1,DrVe07-2,Matejas09}, and inherent parallelism~\cite{EbPa90,BeOVa02,Singer12a,Singer12b,BeOVa15}. Jacobi-type ideas have also been extended to nonsymmetric matrices~\cite{Eberlein62,Veselic76,Veselic79,VeselicWenzel79,Mehl08,BKHari24,BKPe26}, making them relevant for more general eigenvalue problems. However, several key features of the symmetric Jacobi method do not carry over directly to the nonsymmetric setting. In particular, the off-diagonal norm need not decrease monotonically, which makes convergence analysis substantially more delicate.

One such nonsymmetric variant is the ``naive'' Jacobi algorithm. The method is an iterative diagonalization procedure for general square matrices with simple eigenvalues. At each step, two elementary triangular transformations are applied to the current matrix in order to annihilate  selected pivot elements. We have
$$A^{(k+1)}=S_k^{-1}T_k^{-1}A^{(k)}T_kS_k, \quad A^{(0)}=A, \quad k\geq0,$$
where transformations $T_k$ and $S_k$ differ from the identity~$I$ in only one element in upper and lower triangle, respectively. Although naive in the way the annihilations are applied, the method shows high relative accuracy when it converges. 

We prove the asymptotic quadratic convergence of the algorithm applied on almost diagonal matrices. This property was already established in~\cite{Zach}, but here we give a different proof and provide an explicit, verifiable bound for when the quadratic convergence occurs, which was not presented in~\cite{Zach}. This quantitative form is important both theoretically and practically, since it provides a criterion for when the local convergence theory applies. Additionally, we prove the quadratic convergence on almost triangular matrices. Then, we design the preconditioner that moves the starting matrix closer to the region where the quadratic convergence result can be invoked, which is  important as it makes the algorithm applicable to arbitrary complex matrices with simple eigenvalues. 

We implement the preconditioned algorithm in mixed precision arithmetic. Mixed-precision algorithms have become an important topic in numerical linear algebra~\cite{CaHi18,OkCa22,HiMa22}. In general, they aim to provide results of the same quality as an algorithm running in fixed precision, but at a much lower cost. Instead of using only 64-bit IEEE double precision arithmetic~\cite{IEEE}, they combine two or more different floating point precisions. Balancing a lower and a higher precision, the goal is to cut memory bandwidth and energy consumption via lower precision, but maintain the accuracy of the results via higher precision. Following recent evidence that mixed-precision computing can improve the properties of the symmetric Jacobi algorithm~\cite{Higham2025,Zhang2025,Tisseur26}, we construct a low-precision preconditioner, leveraging the reduced computational overhead. The resulting similarity transformation is then applied in double or quadruple precision in order to preserve the spectral information, and the naive Jacobi iteration is carried out in double precision. This way, the low-precision phase is used only where high accuracy is not essential, while the eigenvalue computation itself remains a double-precision procedure. We work with two types of preconditioners, one that moves the matrix near the diagonal form, which is based on the low precision eigendecomposition. It is obtained via the MATLAB \texttt{eig} function performed at low precision. The other type of preconditioner produces a matrix close to the upper-triangular form. This is achieved using the Schur decomposition at low precision, or with several iterations of the QR eigenvalue algorithm performed at low precision~\cite{LAPACK,Golub}.

The key contributions of the paper are as follows:
\begin{itemize}
\item  A new proof of asymptotic quadratic convergence for the naive Jacobi algorithm and the derivation of an explicit, verifiable bound for the convergence region.
\item Convergence of the naive Jacobi algorithm on almost triangular matrices.
\item Two types of low-precision preconditioners that provide practical preconditioning strategy for general matrices with simple eigenvalues.
\item Mixed-precision preconditioned naive Jacobi algorithm that achieves remarkably accurate eigenvalues.
\end{itemize}

In Section~\ref{sec:agm} we present the naive Jacobi algorithm and in Section~\ref{sec:cvg} we prove its asymptotic quadratic convergence, along with the convergence on almost triangular matrices. The main results are given in Theorems~\ref{tm:conv} and~\ref{tm:almosttriang}. Numerical experiments for the algorithm without preconditioning are presented in Section~\ref{sec:num_1}. We move to mixed-precision arithmetic in Section~\ref{sec:mp} where we describe the preconditioning strategy and report the extensive results of the numerical experiments in Section~\ref{sec:num_mp}. Finally, Section~\ref{sec:conclusion} contains concluding remarks.

\section{Detailed description of the algorithm}\label{sec:agm}

We first observe that a general diagonalizable $2\times2$ matrix can be diagonalized in two steps, using two triangular transformations of the form
\begin{equation}\label{TS2}
T=\left[
      \begin{array}{cc}
        1 & x \\
        0 & 1 \\
      \end{array}
    \right] \quad \text{and} \quad S=\left[
      \begin{array}{cc}
        1 & 0 \\
        y & 1 \\
      \end{array}
    \right].
\end{equation}

Let $A\in\C^{2\times 2}$. In~\eqref{TS2}, we choose the parameters $x,y\in\C$  such that
\begin{equation}\label{Tzero}
A'=T^{-1}AT=\left[
      \begin{array}{cc}
        a_{11}' & 0 \\
        a_{21}' & a_{22}' \\
      \end{array}
    \right]
\end{equation}
and
\begin{equation}\label{Szero}
A''=S^{-1}A'S=\left[
      \begin{array}{cc}
        a_{11}'' & 0 \\
        0 & a_{22}'' \\
      \end{array}
    \right].
\end{equation}
After the transformation~\eqref{Tzero}, assuming that $a_{12}\neq0$, we have,
\begin{align*}
A'=T^{-1}AT & = \left[
      \begin{array}{cc}
        1 & -x \\
        0 & 1 \\
      \end{array}
    \right]\left[
      \begin{array}{cc}
        a_{11} & a_{12} \\
        a_{21} & a_{22} \\
      \end{array}
    \right]\left[
      \begin{array}{cc}
        1 & x \\
        0 & 1 \\
      \end{array}
    \right] \\
& = \left[
      \begin{array}{cc}
        a_{11}-a_{21}x & -a_{21}x^2+(a_{11}-a_{22})x+a_{12} \\
        a_{21} & a_{22}+a_{21}x \\
      \end{array}
    \right].
\end{align*}
The condition $a_{12}'=0$ implies
\begin{equation}\label{xeq}
a_{21}x^2+(a_{22}-a_{11})x-a_{12}=0.
\end{equation}
For $a_{21}\neq0$, value $x$ that annihilates $a_{12}$ is a solution of the quadratic equation~\eqref{xeq},
\begin{equation}\label{eq:x12}
x_{1,2}=\frac{(a_{11}-a_{22})\pm\sqrt{(a_{22}-a_{11})^2+4a_{21}a_{12}}}{2a_{21}}.
\end{equation}
We take $x$ as the root of~\eqref{xeq} with a smaller absolute value.
For $a_{21}=0$, equation~\eqref{xeq} is linear with the solution
\begin{equation}\label{eq:xlin}
x=\frac{a_{12}}{a_{22}-a_{11}}, \quad \text{for } a_{11}\neq a_{22}.
\end{equation}
Note that, if $a_{21}=0$ and $a_{11}=a_{22}$, we have
\begin{equation}\label{nondiagonalizable1}
A=\left[
\begin{array}{cc}
 a_{11} & a_{12} \\
 0 & a_{11} \\
\end{array}
\right],
\end{equation}
which means that $A$ is not diagonalizable (for $a_{12}\neq0$), so we can omit this case.
In both cases, $a_{21}\neq0$ or $a_{21}=0$, we get
\begin{equation}\label{appaqq}
a_{11}'=a_{11}-a_{21}x, \quad a_{22}'=a_{22}+a_{21}x, \quad a_{21}'=a_{21}.
\end{equation}

The second transformation acts on $A'$ in the following way,
\begin{align*}
A''=S^{-1}A'S & = \left[
      \begin{array}{cc}
        1 & 0 \\
        -y & 1 \\
      \end{array}
    \right]\left[
      \begin{array}{cc}
        a_{11}' & 0 \\
        a_{21} & a_{22}' \\
      \end{array}
    \right]\left[
      \begin{array}{cc}
        1 & 0 \\
        y & 1 \\
      \end{array}
    \right] \\
& = \left[
      \begin{array}{cc}
        a_{11}' & 0 \\
        (a_{22}'-a_{11}')y+a_{21} & a_{22}' \\
      \end{array}
    \right].
\end{align*}
If $x$ was computed using the formula~\eqref{eq:xlin}, then $a_{21}'=a_{21}=0$ and $y=0$. Otherwise, it follows from~\eqref{Szero} that
$$(a_{22}'-a_{11}')y+a_{21}=0,$$
that is,
\begin{equation}\label{eq:y}
y=\frac{a_{21}}{a_{11}'-a_{22}'}, \quad \text{for } a_{11}'\neq a_{22}'.
\end{equation}
Again, if $a_{11}'=a_{22}'$, we have
\begin{equation}\label{nondiagonalizable2}
A'=\left[
\begin{array}{cc}
 a_{11}' & 0 \\
 a_{21}' & a_{11}' \\
\end{array}
\right],
\end{equation}
meaning that $A'$, as well as $A$, is not diagonalizable (for $a_{21}'=a_{21}\neq0$) and this case is omitted. In conclusion, with $x$ as in~\eqref{eq:x12} or~\eqref{eq:xlin} and $y$ as in~\eqref{eq:y}, we get
$$S^{-1}T^{-1}ATS=\left[
      \begin{array}{cc}
        \lambda_1 & 0 \\
        0 & \lambda_2 \\
      \end{array}
    \right].$$
Using this idea, an iterative algorithm, first called \emph{naive Jacobi method} in~\cite{Zach}, is constructed for the diagonalization of $n\times n$ matrices.

Let $A\in\C^{n\times n}$ be a diagonalizable matrix. One iteration of the naive Jacobi algorithm takes the form
\begin{equation}\label{Jacobi_iteration}
A^{(k+1)}=S_k^{-1}T_k^{-1}A^{(k)}T_kS_k, \quad A^{(0)}=A, \quad k\geq0,
\end{equation}
where $\displaystyle  T_k=I+x_ke_{p_k}e_{q_k}^T$ and $S_k=I+y_ke_{q_k}e_{p_k}^T$, for the standard basis vectors $e_{p_k}$ and $e_{q_k}$. That is, $T_k$ and $S_k$ are $n\times n$ matrices differing from the identity in only one element, $x_k$ on the position $(p_k,q_k)$ and $y_k$ on the position $(q_k,p_k)$, respectively.
Iterations~\eqref{Jacobi_iteration} can also be written as
$$A^{(k+1)}=V_k^{-1}A^{(k)}V_k, \quad A^{(0)}=A, \quad k\geq0,$$
where $V_k=T_kS_k$ differs from the identity only in the $2\times2$ submatrix
$$\hat{V}_k=\left[\begin{array}{cc}
1+x_k y_k & x_k \\
y_k & 1 \\
\end{array}\right],$$
obtained at the intersection of the $p_k$th and $q_k$th row and column. Such a submatrix is called the \emph{pivot submatrix} defined by the \emph{pivot pair} $(p_k,q_k)$, $1\leq p_k<q_k\leq n$.

In the $(k+1)$th step of~\eqref{Jacobi_iteration}, the pivot submatrix
$$\hat{A}^{(k)}=\left[
      \begin{array}{cc}
        a_{p_kp_k}^{(k)} & a_{p_kq_k}^{(k)} \\
        a_{q_kp_k}^{(k)} & a_{q_kq_k}^{(k)} \\
      \end{array}
    \right]$$
is diagonalized. We have
$$\hat{A}^{(k+1)}=\hat{S}_k^{-1}\hat{T}_k^{-1}\hat{A}^{(k)}\hat{T}_k\hat{S}_k
=\left[
      \begin{array}{cc}
        a_{p_kp_k}^{(k+1)} & 0 \\
        0 & a_{q_kq_k}^{(k+1)} \\
      \end{array}
    \right],$$
where $\hat{T}_k$ and $\hat{S}_k$ are of the form given in~\eqref{TS2}. Transformation parameters $x_k$ and $y_k$ are analogous to those for $2\times2$ matrices, just taking the index pair $(p_k,q_k)$ instead of $(1,2)$. For $(p,q)=(p_k,q_k)$, one iteration~\eqref{Jacobi_iteration} changes only $p$th and $q$th column and row of $A^{(k)}$. Precisely,
\begin{equation}\label{elements}
\begin{aligned}
a_{pq}^{(k+1)} & =a_{qp}^{(k+1)}=0, \\
a_{pp}^{(k+1)} & =a_{pp}^{(k)}-x_ka_{qp}^{(k)}, \\
a_{qq}^{(k+1)} & =a_{qq}^{(k)}+x_ka_{qp}^{(k)}, \\
a_{ip}^{(k+1)} & =(1+x_ky_k)a_{ip}^{(k)}+y_ka_{iq}^{(k)}, \quad i\neq p,q,\\
a_{pi}^{(k+1)} & = a_{pi}^{(k)}-x_ka_{qi}^{(k)}, \quad i\neq p,q,\\
a_{qi}^{(k+1)} & =(1+x_ky_k)a_{qi}^{(k)}-y_ka_{pi}^{(k)},\quad  i\neq p,q,\\
a_{iq}^{(k+1)} & =a_{iq}^{(k)}+x_ka_{ip}^{(k)}, \quad i\neq p,q, \\
a_{ij}^{(k+1)} & =a_{ij}^{(k)}, \quad i,j\neq p,q.
\end{aligned}
\end{equation}

The order of the pivot pairs defines the pivot strategy. We consider cyclic pivot strategies. They take in a prescribed order all possible pivot pairs, which are those from the upper triangle of an $n\times n$ matrix, and repeat that order of the pivot pairs, until convergence. To give an example, a well-known cyclic pivot strategy is a row-wise strategy that repeats pivot position in the following order:
\begin{equation}\label{row-wise}
(1,2),(1,3),\ldots,(1,n),(2,3),\ldots,(2,n),\ldots,(n-1,n).
\end{equation}

The annihilation of the pivot elements is repeated cyclically until the stopping criterion is satisfied, usually until the off-diagonal norm of the underlying matrix becomes smaller than the prescribed tolerance. Then, the diagonal elements of $\Lambda=A^{(K)}$ are taken as the eigenvalues of $A$, while the columns of the non-singular matrix 
\begin{equation}\label{transformationV}
V= V_0V_1\cdots V_{K-1}=T_0S_0T_1S_1\cdots  T_{K-1}S_{K-1}
\end{equation}
are the computed eigenvectors of $A$. 

In order to reduce the number of operations, when computing the parameter $x_k$, we do not compute  both solutions of the quadratic equation~\eqref{eq:x12}. Instead, we set
$$d_k=a_{qq}^{(k)}-a_{pp}^{(k)}, \quad s_k=\sqrt{d_k^2+4a_{qp}^{(k)}a_{pq}^{(k)}},$$
choose $\sigma_k=\pm1$ such that $|d_k+\sigma_k s_k|$ is maximal, and compute
\begin{equation}\label{eq:ximplementation}
x_k=\frac{2a_{pq}^{(k)}}{d_k+\sigma_k s_k}.
\end{equation}
This way we need to compute the square root only once per iteration. One can check relation~\eqref{eq:x12lemma} from the proof of Lemma~\ref{tm:xy} for detailed derivation of the expression~\eqref{eq:ximplementation}. Moreover, we can compute the parameter $y_k$ as 
\begin{equation}\label{eq:yimplementation}
y_k=-\sigma_k\frac{a_{qp}^{(k)}}{s_k}.
\end{equation}
See equation~\eqref{eq:ylemma} from the proof of Lemma~\ref{tm:xy} for details.

We should also note that, if $|a_{p_kq_k}^{(k)}|$ and $|a_{q_kp_k}^{(k)}|$ are both very close to zero, smaller than a prescribed tolerance, pivot pair $(p_k,q_k)$ is skipped, that is, $T_k=S_k=I$, and we move to $(p_{k+1},q_{k+1})$. Also, in order to reduce the round-off error, the updates of the diagonal elements can be kept in memory and applied at the same time at the end of a cycle.

The discussion presented in this section is summarized in Algorithm~\ref{agm:Jacobi}.

\begin{algorithm}[p]
\caption{Naive Jacobi algorithm}\label{agm:Jacobi}
\renewcommand{\algorithmicrequire}{\textbf{Input:}}
\renewcommand{\algorithmicensure}{\textbf{Output:}}
\begin{algorithmic}
\Require $A\in\C^{n\times n}$ \text{diagonalizable matrix with simple eigenvalues $\lambda_i^*$, $i=1,\ldots,n$}
\Ensure An approximate spectral decomposition $A\approx V\Lambda V^{-1}$
\State $V=I$;
\Repeat
\State Take $(p,q)$ according to the pivot strategy.
\State \textbf{if} $a_{pq}=a_{qp}=0$ \textbf{then} skip this pivot pair. \Comment{Submatrix already diagonal}
\State \textbf{if} $a_{pq}a_{qp}=0$ and $a_{pp}=a_{qq}$ \textbf{then} skip this pivot pair. \Comment{Nondiag. submatrix, \eqref{nondiagonalizable1}}
\If {$a_{pq}\neq0$}
\If {$a_{qp}\neq0$}
\State $d=a_{qq}-a_{pp}$;
\State $s=\sqrt{d^2+4a_{qp}a_{pq}}$;
\If {$|d+s|>|d-s|$}
\State $x=2a_{pq}/(d+s)$; \Comment{\eqref{eq:ximplementation}}
\State \textbf{if} $a_{pp}-a_{qq}=2xa_{qp}$ \textbf{then} skip this pivot pair. \Comment{Nondiag. submatrix, \eqref{nondiagonalizable2}}
\State \textbf{else} $y=-a_{qp}/s$; \textbf{end if} \Comment{\eqref{eq:yimplementation}}
\Else
\State $x=2a_{pq}/(d-s)$; \Comment{\eqref{eq:ximplementation}}
\State \textbf{if} $a_{pp}-a_{qq}=2xa_{qp}$ \textbf{then} skip this pivot pair. \Comment{Nondiag. submatrix, \eqref{nondiagonalizable2}}
\State \textbf{else} $y=a_{qp}/s$; \textbf{end if}\Comment{\eqref{eq:yimplementation}}
\EndIf
\Else
\State $x=a_{pq}/(a_{qq}-a_{pp})$; \Comment{\eqref{eq:xlin}}
\State $y=0$;
\EndIf
\Else
\State $x=0$;
\State $y=a_{qp}/(a_{pp}-a_{qq})$; \Comment{\eqref{eq:y}}
\EndIf
\State
\State $a_{pp}=a_{pp}-xa_{qp}$; \  $a_{qq}=a_{qq}+xa_{qp}$;
\State $a_{pq}=0$; \ $a_{qp}=0$;
\State $t=1+xy$;
\For {$i=1:p-1,p+1:q-1,q+1:n$}
\State $temp1=a_{pi}-xa_{qi}$;
\State $a_{qi}=-ya_{pi}+ta_{qi}$;
\State $a_{pi}=temp1$;
\State $temp2=ta_{ip}+ya_{iq}$;
\State $a_{iq}=xa_{ip}+a_{iq}$;
\State $a_{ip}=temp2$; \Comment{\eqref{elements}}
\EndFor
\For {$i=1:n$}
\State $temp3=v_{iq}+xv_{ip}$;
\State $v_{ip}=v_{ip}+ytemp3$;
\State $v_{iq}=temp3$;\Comment{\eqref{transformationV}}
\EndFor
\Until convergence
\end{algorithmic}
\end{algorithm}

All operations in Algorithm~\ref{agm:Jacobi} are performed in complex floating-point arithmetic. According to the standard LAPACK convention, one complex addition or subtraction is counted as two real FLOPs (floating point operations), one complex multiplication as six real FLOPs, and one complex division as 16 real FLOPs. The complex square root is the only operation whose cost is not standardized; depending on the implementation, it typically requires 20--40 real FLOPs. Since only one square root is computed per iteration, it does not significantly affect the overall computational cost of one cycle of the naive Jacobi algorithm, which is $30n^3 + O(n^2)$ FLOPs. If only the eigenvalues are computed the FLOP count decreases to $22n^3 + O(n^2)$.

\section{Asymptotic convergence}\label{sec:cvg}

In Subsection~\ref{sec:diagonal}, we prove the asymptotic quadratic convergence of the Algorithm~\ref{agm:Jacobi} for almost diagonal matrices with distinct eigenvalues. A similar result was also proven by Zacharias~\cite[Satz~2.7, p.~29]{Zach}, but our approach is different. Most importantly, in Theorem~\ref{tm:conv} we provide a verifiable convergence criterion that includes the diagonal elements of the starting matrix, while the criterion in~\cite{Zach} is purely theoretical, since it involves unknown eigenvalues of the starting matrix~\cite[Eq.~(2.1), p.~15]{Zach}. 

Additionally, in Subsection~\ref{sec:triangular}, we show that the naive Jacobi converges on almost triangular matrices.

\subsection{Almost diagonal matrices}\label{sec:diagonal}

Denote the smallest distance between two diagonal elements of $A^{(k)}=(a_{ij}^{(k)})$ by
\begin{equation}\label{mu_k}
\mu_k\coloneqq \min_{i\neq j}|a_{ii}^{(k)}-a_{jj}^{(k)}|,
\end{equation}
and the largest modulus of the off-diagonal element of $A^{(k)}$ by
\begin{equation}\label{eta_k}
\eta_k\coloneqq \max_{i\neq j}|a_{ij}^{(k)}|.
\end{equation}
In Lemma~\ref{tm:xy} and Corollaries~\ref{tm:maxoffdiag} and~\ref{tm:mueta} we assume that
\begin{equation}\label{assumption_conv}
\mu_k>0 \quad \text{and} \quad \frac{\eta_k}{\mu_k}\leq\frac{1}{2\sqrt{2}}, 
\end{equation}
for some $k$. The first relation is a separation assumption on the diagonal entries, while the second one reflects the asymptotic stage. These assumptions ensure that the transformation parameters $x_k$ and $y_k$ introduced in the previous section are well defined. In Lemma~\ref{tm:xy} we bound their values (cf. \cite[Lemma~2.2]{Zach}).

\begin{Lemma}\label{tm:xy}
If assumptions~\eqref{assumption_conv} hold, then
\begin{equation}\label{eq:xy}
|x_k|\leq2\frac{\eta_k}{\mu_k} \quad \text{and} \quad |y_k|\leq\sqrt{2}\frac{\eta_k}{\mu_k}.
\end{equation}
\end{Lemma}

\begin{proof}
For a fixed iteration step $k$, let $(p,q)=(p_k,q_k)$. If $a_{qp}^{(k)}=0$, it follows directly from the relation~\eqref{eq:xlin} and the definitions of $\mu_k$ and $\eta_k$, given by~\eqref{mu_k} and~\eqref{eta_k}, respectively, that $|x_k|\leq\frac{\eta_k}{\mu_k}$.

Let $a_{qp}^{(k)}\neq0$. Parameter $x_k$ is obtained as in the relation~\eqref{eq:x12} and it can be written as
\begin{align*}
x_k & =\frac{\left((a_{pp}^{(k)}-a_{qq}^{(k)})\pm\sqrt{(a_{qq}^{(k)}-a_{pp}^{(k)})^2+4a_{qp}^{(k)}a_{pq}^{(k)}}\right)\left((a_{pp}^{(k)}-a_{qq}^{(k)})\mp\sqrt{(a_{qq}^{(k)}-a_{pp}^{(k)})^2+4a_{qp}^{(k)}a_{pq}^{(k)}}\right)}{2a_{qp}^{(k)}\left((a_{pp}^{(k)}-a_{qq}^{(k)})\mp\sqrt{(a_{qq}^{(k)}-a_{pp}^{(k)})^2+4a_{qp}^{(k)}a_{pq}^{(k)}}\right)} \\
& = \frac{-4a_{qp}^{(k)}a_{pq}^{(k)}}{2a_{qp}^{(k)}\left((a_{pp}^{(k)}-a_{qq}^{(k)})\mp\sqrt{(a_{qq}^{(k)}-a_{pp}^{(k)})^2+4a_{qp}^{(k)}a_{pq}^{(k)}}\right)} \\
& = \frac{2a_{pq}^{(k)}}{(a_{qq}^{(k)}-a_{pp}^{(k)})\pm\sqrt{(a_{qq}^{(k)}-a_{pp}^{(k)})^2+4a_{qp}^{(k)}a_{pq}^{(k)}}}.
\end{align*}
Denote
\begin{equation}\label{eq:ds}
d_k\coloneqq a_{qq}^{(k)}-a_{pp}^{(k)}, \quad s_k\coloneqq \sqrt{d_k^2+4a_{qp}^{(k)}a_{pq}^{(k)}},
\end{equation} 
where the same branch of the square root is chosen as in the definition of $x_k$. The assumptions~\eqref{assumption_conv} imply $d_k\neq 0$ and $s_k\neq 0$.
Then
\begin{equation}\label{eq:x12lemma}
x_k=\frac{2a_{pq}^{(k)}}{d_k\pm s_k}.
\end{equation}
Since $x_k$ is chosen as the root with the smaller modulus, it is given by the fraction~\eqref{eq:x12lemma} with  the denominator of the larger modulus. Therefore,
$$|x_k|\leq\frac{2|a_{pq}^{(k)}|}{\max\{|d_k+s_k|,|d_k-s_k|\}}.$$
We observe that
$$|d_k|=\frac{1}{2}|(d_k+s_k)+(d_k-s_k)|\leq\frac{1}{2}\left(|d_k+s_k|+|d_k-s_k|\right)\leq\max\{|d_k+s_k|,|d_k-s_k|\}.$$
Hence, by the definitions of $\eta_k$ and $\mu_k$, we have
$$|x_k|\leq \frac{2|a_{pq}^{(k)}|}{|d_k|}\leq2\frac{\eta_k}{\mu_k}.$$

On the other hand, from~\eqref{eq:y} and~\eqref{appaqq} we have
$$y_k=\frac{a_{qp}^{(k)}}{a_{pp}^{(k)}-a_{qq}^{(k)}-2a_{qp}^{(k)}x_k}.$$
Using  equation~\eqref{eq:x12} for $x_k$, we get
$$2a_{qp}^{(k)}x_k=(a_{pp}^{(k)}-a_{qq}^{(k)})\pm\sqrt{(a_{qq}^{(k)}-a_{pp}^{(k)})^2+4a_{qp}^{(k)}a_{pq}^{(k)}}=-d_k\pm s_k,$$
and
\begin{equation}\label{eq:ylemma}
y_k=\frac{a_{qp}^{(k)}}{\mp s_k}.
\end{equation}
It follows from~\eqref{eq:ds} and the reverse triangular inequality that
$$|s_k|=|d_k|\left|\sqrt{1+\frac{4a_{qp}^{(k)}a_{pq}^{(k)}}{d_k^2}}\right|=|d_k|\sqrt{\left|1+\frac{4a_{qp}^{(k)}a_{pq}^{(k)}}{d_k^2}\right|} \geq |d_k|\sqrt{1-\left|\frac{4a_{qp}^{(k)}a_{pq}^{(k)}}{d_k^2}\right|},$$
since, by~\eqref{mu_k}, \eqref{eta_k} and the assumption~\eqref{assumption_conv} we have
$$\left|\frac{4a_{qp}^{(k)}a_{pq}^{(k)}}{d_k^2}\right|\leq\frac{4\eta_k^2}{\mu_k^2}\leq\frac{1}{2},$$
which implies
$$|s_k|\geq\frac{1}{\sqrt{2}}|d_k|.$$
This, together with~\eqref{eq:ylemma}, gives
$$|y_k|\leq\sqrt{2}\frac{|a_{qp}^{(k)}|}{|d_k|}\leq\sqrt{2}\frac{\eta_k}{\mu_k}.$$
\end{proof}

In the $k$th step of the algorithm, for the pivot position $(p,q)=(p_k,q_k)$, only the elements in the $p$th and $q$th column and row of $A^{(k)}$ are changed. Corollary~\ref{tm:maxoffdiag} provides the upper bound on the growth of the off-diagonal elements.

\begin{Corollary}\label{tm:maxoffdiag}
If the assumptions~\eqref{assumption_conv} hold, then
\begin{equation}\label{eq:maxoffdiag}
\begin{aligned}
& |a_{pi}^{(k+1)}|\leq|a_{pi}^{(k)}|+2\frac{\eta_k^2}{\mu_k}, &\quad |a_{ip}^{(k+1)}|\leq|a_{ip}^{(k)}|+(1+\sqrt{2})\frac{\eta_k^2}{\mu_k}, \\
& |a_{iq}^{(k+1)}|\leq|a_{iq}^{(k)}|+2\frac{\eta_k^2}{\mu_k}, & \quad |a_{qi}^{(k+1)}|\leq|a_{qi}^{(k)}|+(1+\sqrt{2})\frac{\eta_k^2}{\mu_k},
\end{aligned}
\end{equation}
for $(p,q)=(p_k,q_k)$, $i\neq p,q$.
\end{Corollary}

\begin{proof}
Using the relations~\eqref{elements}, \eqref{eq:xy}, and~\eqref{assumption_conv} we get
\begin{align*}
|a_{pi}^{(k+1)}| & \leq |a_{pi}^{(k)}|+|x_k||a_{qi}^{(k)}| \leq |a_{pi}^{(k)}|+2\frac{\eta_k}{\mu_k}\eta_k = |a_{pi}^{(k)}|+2\frac{\eta_k^2}{\mu_k}, \\
|a_{iq}^{(k+1)}| & \leq |a_{iq}^{(k)}|+|x_k||a_{ip}^{(k)}| \leq |a_{iq}^{(k)}|+2\frac{\eta_k^2}{\mu_k}, \\
|a_{ip}^{(k+1)}| & \leq |a_{ip}^{(k)}|+|x_k||y_k||a_{ip}^{(k)}|+|y_k||a_{iq}^{(k)}| \leq |a_{ip}^{(k)}|+2\sqrt{2}\frac{\eta_k^2}{\mu_k^2}\eta_k+\sqrt{2}\frac{\eta_k}{\mu_k}\eta_k \\
& \leq |a_{ip}^{(k)}|+\frac{\eta_k^2}{\mu_k}+\sqrt{2}\frac{\eta_k^2}{\mu_k} = |a_{ip}^{(k)}|+(1+\sqrt{2})\frac{\eta_k^2}{\mu_k}, \\
|a_{qi}^{(k+1)}| & \leq |a_{qi}^{(k)}|+|x_k||y_k||a_{qi}^{(k)}|+|y_k||a_{pi}^{(k)}| \leq |a_{qi}^{(k)}|+(1+\sqrt{2})\frac{\eta_k^2}{\mu_k}.
\end{align*}
\end{proof}

In one step of the algorithm, the largest absolute value of an off-diagonal element, $\eta_k$, can increase, while the smallest distance between two diagonal elements, $\mu_k$, can decrease. These are not desirable properties, but increase in $\eta_k$ and decrease in $\mu_k$ is bounded. This is shown in Corollary~\ref{tm:mueta}.

\begin{Corollary}\label{tm:mueta}
If the assumptions~\eqref{assumption_conv} hold, then
\begin{align}
\eta_{k+1} & \leq2\eta_k, \label{eq:eta_leq}\\
\mu_{k+1} & \geq \mu_k-\sqrt{2}\eta_k. \nonumber
\end{align}
\end{Corollary}

\begin{proof}
Inequality~\eqref{eq:eta_leq} follows directly from the definition of $\eta_k$ and the Corollary~\ref{tm:maxoffdiag}, using~\eqref{assumption_conv}. For any index pair $(i,j)$, $i\neq j$, we have
$$|a_{ij}^{(k+1)}|\leq \eta_k+(1+\sqrt{2})\frac{\eta_k^2}{\mu_k}\leq\left(1+\frac{1+\sqrt{2}}{2\sqrt{2}}\right)\eta_k<2\eta_k.$$

From the relations~\eqref{elements} we have
$$a_{ii}^{(k+1)}=a_{ii}^{(k)}+\Delta_i, \quad 1\leq i\leq n,$$
where $\Delta_p=-x_ka_{qp}^{(k)}$, $\Delta_q=x_ka_{qp}^{(k)}$, and $\Delta_i=0$ otherwise.
Lemma~\ref{tm:xy} and assumptions~\eqref{assumption_conv} imply
$$|\Delta_i|\leq\frac{1}{\sqrt{2}}\eta_k, \quad 1\leq i\leq n.$$
Using the triangle inequality, we get
$$|a_{ii}^{(k)}-a_{jj}^{(k)}|=|a_{ii}^{(k+1)}-a_{jj}^{(k+1)}-\Delta_i+\Delta_j|\leq|a_{ii}^{(k+1)}-a_{jj}^{(k+1)}|+|\Delta_i|+|\Delta_j|,$$
that is
$$|a_{ii}^{(k+1)}-a_{jj}^{(k+1)}|\geq|a_{ii}^{(k)}-a_{jj}^{(k)}|-|\Delta_i|-|\Delta_j|, \quad 1\leq i,j\leq n.$$
Thus, if $\mu_{k+1}=|a_{ii}^{(k+1)}-a_{jj}^{(k+1)}|$, then $\mu_{k+1}\geq\mu_k-\frac{2}{\sqrt{2}}\eta_k=\mu_k-\sqrt{2}\eta_k$.
\end{proof}

Before we prove the main result of this section, we need one additional lemma.

\begin{Lemma}\label{tm:convpom}
Let $A^{(k)}$ be generated by Algorithm~\ref{agm:Jacobi} under a fixed cyclic pivot strategy. Let $\mu_k$ and $\eta_k$, $k\geq0$, be as given in~\eqref{mu_k} and~\eqref{eta_k}. Let $n\geq 3$, $\mu_0>0$, and
\begin{equation}\label{eq:tmassumption}
80(n-1)\frac{\eta_0}{\mu_0}\leq1.
\end{equation}
Then, for $N=n(n-1)/2$, the following inequalities hold:
\begin{subequations}\label{eq:pom}
\begin{align}
\eta_{kN} & \leq \frac{\eta_0}{2^k}, \label{eq:pom_etakN} \\
\mu_{kN} & \geq \frac{\mu_0}{2}. \label{eq:pom_mukN} 
\end{align}
\end{subequations}
\end{Lemma}

\begin{proof}
We prove the inequalities~\eqref{eq:pom} using mathematical induction. 

For $k=0$, the claim is immediate. 
Assume that~\eqref{eq:pom} hold for some fixed $k\geq0$. Then, since $\mu_{kN}>0$,
\begin{equation}\label{eq:pom_eta/mu}
\frac{\eta_{kN}}{\mu_{kN}}\leq\frac{\eta_0}{\mu_0}.
\end{equation}
We first show that, besides~\eqref{eq:pom} and~\eqref{eq:pom_eta/mu}, inequalities 
\begin{equation}\label{eq:pom_p}
\eta_{kN+l}\leq2\eta_{kN} \quad \text{and} \quad \mu_{kN+l}\geq\frac{\mu_{kN}}{2}
\end{equation}
hold for this $k$ and for every $0\leq l<N$. We prove this claim by strong induction over $l$. 

For $l=0$, \eqref{eq:pom_p} is obvious. 
We assume that~\eqref{eq:pom_p} holds  for all $r$, $0\leq r\leq l<N-1$. Moreover, since $\mu_{kN+r}>0$, we have
$$\frac{\eta_{kN+r}}{\mu_{kN+r}}\leq4\frac{\eta_{kN}}{\mu_{kN}}\substack{\eqref{eq:pom_eta/mu} \\ \leq} 4\frac{\eta_0}{\mu_0} \substack{\eqref{eq:tmassumption} \\ \leq} \frac{1}{20(n-1)} <\frac{1}{2\sqrt{2}},$$
which ensures that the assumptions~\eqref{assumption_conv} are satisfied at each step $kN+r$, $0\leq r\leq l$.

We observe an off-diagonal position $(i,j)$. During one cycle, the entry at position $(i,j)$ is annihilated when $(i,j)$ is the pivot pair and modified only when the pivot pair shares one index with $(i,j)$, which happens in $2(n-2)$ steps. These changes are bounded by the Corollary~\ref{tm:maxoffdiag}.
Thus, if $(i,j)$-entry was annihilated in the first $l$ steps of the cycle, using the bounds~\eqref{eq:maxoffdiag} and the assumption~\eqref{eq:pom_p}, we get
\begin{equation}\label{eq:aij_pivot}
|a_{ij}^{(kN+l+1)}|\leq2(n-2)(1+\sqrt{2})\frac{4\eta_{kN}^2}{\mu_{kN}/2}<40(n-2)\frac{\eta_{kN}^2}{\mu_{kN}} \substack{\eqref{eq:pom_eta/mu} \\ \leq} 40(n-2)\frac{\eta_0}{\mu_0}\eta_{kN} \substack{\eqref{eq:tmassumption} \\ \leq} \eta_{kN}.
\end{equation}
If it was not annihilated during the first $l$ steps, then
$$|a_{ij}^{(kN+l+1)}|\leq|a_{ij}^{(kN)}|+2(n-2)(1+\sqrt{2})\frac{4\eta_{kN}^2}{\mu_{kN}/2} \substack{\eqref{eq:aij_pivot} \\ \leq} |a_{ij}^{(kN)}|+\eta_{kN} \leq 2\eta_{kN}.$$
Since the index pair $(i,j)$ was arbitrary, this proves the first relation in~\eqref{eq:pom_p}.

The diagonal entry at the position $(i,i)$, $1\leq i\leq n$, changes if one of the indices from the pivot pair equals $i$. During one cycle, that happens $n-1$ times. The change can be bounded using the relations~\eqref{elements}, Lemma~\ref{tm:xy}, and the assumption~\eqref{eq:pom_p} for $0\leq r\leq l$. We have
$$|a_{ii}^{(kN+r+1)}-a_{ii}^{(kN+r)}|\leq|x_{kN+r}||a_{ji}^{(kN+r)}|\substack{\eqref{eq:xy} \\ \leq} 2\frac{\eta_{kN+r}}{\mu_{kN+r}}\eta_{kN+r} \substack{\eqref{eq:pom_p} \\ \leq} 16\frac{\eta_{kN}^2}{\mu_{kN}}.$$
After $l+1$ steps of a cycle, this comes to
\begin{equation}\label{eq:pom_aii}
|a_{ii}^{(kN+l+1)}-a_{ii}^{(kN)}|\leq 16(n-1)\frac{\eta_{kN}^2}{\mu_{kN}}.
\end{equation}
Using the triangle inequality we obtain
\begin{align*}
|a_{ii}^{(kN)}-a_{jj}^{(kN)}| \leq |a_{ii}^{(kN)}-a_{ii}^{(kN+l+1)}|+|a_{ii}^{(kN+l+1)}-a_{jj}^{(kN+l+1)}|+|a_{jj}^{(kN+l+1)}-a_{jj}^{(kN)}|,
\end{align*}
that is,
\begin{align*}
|a_{ii}^{(kN+l+1)}-a_{jj}^{(kN+l+1)}| \geq |a_{ii}^{(kN)}-a_{jj}^{(kN)}|-|a_{ii}^{(kN+l+1)}-a_{ii}^{(kN)}|-|a_{jj}^{(kN+l+1)}-a_{jj}^{(kN)}|.
\end{align*}
Hence,
\begin{align}
\mu_{kN+l+1} & \substack{\eqref{eq:pom_aii} \\ \geq}\mu_{kN}-32(n-1)\frac{\eta_{kN}^2}{\mu_{kN}}=\mu_{kN}-32(n-1)\frac{\eta_{kN}^2}{\mu_{kN}^2}\mu_{kN} \label{eq:mukp1} \\
& \substack{\eqref{eq:pom_eta/mu} \\ \geq} \mu_{kN}-32(n-1)\frac{\eta_0^2}{\mu_0^2}\mu_{kN} \substack{\eqref{eq:tmassumption} \\ \geq} \mu_{kN}-\frac{1}{200(n-1)}\mu_{kN}>\frac{\mu_{kN}}{2}, \nonumber
\end{align}
which proves the second relation in~\eqref{eq:pom_p}.

Now, we show that the inequalities~\eqref{eq:pom} also hold for $k+1$. As it was observed earlier in this proof, during one cycle, off-diagonal entry at the position $(i,j)$ is annihilated exactly once and, after annihilation, it is changed at most $2(n-2)$ times. Therefore, repeating the calculation done in~\eqref{eq:aij_pivot}, we obtain
\begin{equation}\label{eq:aij}
|a_{ij}^{(k+1)N}| \leq 40(n-2)\frac{\eta_0}{\mu_0}\eta_{kN}<40(n-1)\frac{\eta_0}{\mu_0}\eta_{kN}.
\end{equation}
This holds for any index pair $(i,j)$. Thus, the assumption~\eqref{eq:tmassumption} implies
$$\eta_{(k+1)N} \leq \frac{1}{2}\eta_{kN}.$$
Now, it follows from the assumption~\eqref{eq:pom_etakN} for $k$ that
$$\eta_{(k+1)N}\leq\frac{1}{2}\frac{\eta_0}{2^k}=\frac{\eta_0}{2^{k+1}},$$
that is, the bound~\eqref{eq:pom_etakN} holds for $k+1$.

For the bound on $\mu_{(k+1)N}$, we use the calculation done in~\eqref{eq:mukp1},
$$\mu_{(k+1)N}\geq \mu_{kN}-32(n-1)\frac{\eta_{kN}^2}{\mu_{kN}}.$$
Applying this recursively for the cycles $k,k-1,\dots,0$ and using the induction hypothesis, we get
\begin{align*}
\mu_{(k+1)N} & \geq \mu_0-32(n-1)\sum_{j=0}^k\frac{\eta_{jN}^2}{\mu_{jN}} \substack{\eqref{eq:pom_etakN} \eqref{eq:pom_mukN}\\ \geq} \mu_0-32(n-1)\sum_{j=0}^k\frac{(\eta_0/2^j)^2}{\mu_0/2} \\ 
& = \mu_0-\frac{64(n-1)\eta_0^2}{\mu_0}\sum_{j=0}^k 4^{-j} \geq \mu_0-\frac{64(n-1)\eta_0^2}{\mu_0}\sum_{j=0}^{\infty}4^{-j} \\
& =\mu_0-\frac{256(n-1)\eta_0^2}{3\mu_0} =\mu_0-\frac{256(n-1)\eta_0^2}{3\mu_0^2}\mu_0,
\end{align*}
because $\sum_{j=0}^{\infty}4^{-j}=\frac{4}{3}$. Then, assumption~\eqref{eq:tmassumption} with $n\geq2$ implies
$$\mu_{(k+1)N}\geq\mu_0-\frac{0.02}{n-1}\mu_0>\frac{\mu_0}{2},$$
and we conclude that~\eqref{eq:pom_mukN} holds for $k+1$.
\end{proof}

If a matrix $A\in\C^{n\times n}$ is sufficiently close to a diagonal matrix, then its diagonal entries approximate the eigenvalues of $A$. In the next proposition (cf.~\cite[Lemma 2.5]{Zach}), we show that, at an advanced stage of the naive Jacobi method, each diagonal entry approximates a unique eigenvalue of $A$.

\begin{Proposition}\label{prop:eig}
Let $n\geq 3$, $k\geq0$, and the assumption~\eqref{eq:tmassumption} hold.
If an eigenvalue $\lambda_p^*$ of $A$ lies within the $p$th Gershgorin disk $G_p^{(kN)}$ for some $k$ and for $1\leq p\leq n$, then $(a_{pp}^{(kN)})$ converges to~$\lambda_p^*.$
\end{Proposition}

\begin{proof}
By the assumption on $\lambda_p^*$, and by~\eqref{eq:pom_etakN}we have
$$|a_{pp}^{(kN)}-\lambda_p^*| \leq \sum_{j\neq p}|a_{pj}^{(kN)}|\leq (n-1) \eta_{kN}\leq (n-1)\frac{\eta_0}{2^k}.$$
Letting $k\to\infty$, we get 
$$\lim_{k\to\infty} a_{pp}^{(kN)}=\lambda_p^*.$$
\end{proof}

The following remark is implied by the previous discussion. The matrix off-norm is, as usual, defined as
$$\off^2(A)=\sum_{i\neq j}|a_{ij}|^2.$$

\begin{Remark}
Relation~\eqref{eq:pom_etakN} implies that $\lim_{k\rightarrow\infty}\eta_{kN}=0$, that is, under assumption~\eqref{eq:tmassumption}, the sequence of the off-norms $\off(A^{(kN)})$ converges to zero. Additionally, Proposition~\ref{prop:eig} ensures that the sequence $(A^{(kN)})$ converges to a diagonal matrix with some fixed order of the eigenvalues on its diagonal entries. \end{Remark}

Using the results from Lemma~\ref{tm:convpom}, we are ready to prove that  convergence of $\eta_{kN}$ is quadratic. 

\begin{Theorem}\label{tm:conv}
Let $A=(a_{ij})\in\mathbb{C}^{n\times n}$, $n\geq 3$, and let $A^{(k)}=(a_{ij}^{(k)})$ denote the matrix obtained from $A$ after performing $k$ iterations of the form~\eqref{Jacobi_iteration} under a fixed cyclic pivot strategy. Let $\mu_k$ and $\eta_k$, $k\geq0$, be as given in~\eqref{mu_k} and~\eqref{eta_k}, with $\mu_0>0$ and
$$80(n-1)\frac{\eta_0}{\mu_0}\leq1.$$
Then, iteration process~\eqref{Jacobi_iteration} converges quadratically, that is, for $N=\frac{n(n-1)}{2}$,
$$\eta_{(k+1)N}\leq C\eta_{kN}^2,$$
and
$$\off\left(A^{((k+1)N)}\right)\leq C\sqrt{2N}\off^2\left(A^{(kN)}\right),$$
where $C\coloneqq\frac{80(n-1)}{\mu_0}$ is a constant depending only on $n$ and $\mu_0$.
\end{Theorem}

\begin{proof}
Assumptions of the theorem are the same as the assumptions of Lemma~\ref{tm:convpom}. By repeating the calculation done in~\eqref{eq:aij_pivot} and~\eqref{eq:aij}, and using the bound~\eqref{eq:pom_mukN}, we get
$$|a_{ij}^{(k+1)N}| \leq 40(n-1)\frac{\eta_{kN}^2}{\mu_{kN}}\leq80(n-1)\frac{\eta_{kN}^2}{\mu_0}.$$
Therefore, 
\begin{equation}\label{eq:tmquad}
\eta_{(k+1)N}\leq\frac{80(n-1)}{\mu_0}\eta_{kN}^2=C\eta_{kN}^2,
\end{equation}
for $C$ as in the statement of the theorem.

Definition of the off-norm yields the bounds
$$\eta_k^2\leq\off^2\left(A^{(k)}\right)\leq2N\eta_k^2.$$
Combining this with~\eqref{eq:tmquad}, we obtain
$$\off\left(A^{((k+1)N)}\right) \leq \sqrt{2N}\eta_{(k+1)N} \leq C\sqrt{2N}\eta_{kN}^2 \leq C\sqrt{2N}\off^2\left(A^{(kN)}\right).$$
\end{proof}

\subsection{Almost triangular matrices}\label{sec:triangular}

Note that the results for almost diagonal matrices hold for any cyclic pivot strategy. For the (almost) triangular matrices we specifically observe the row-wise pivot strategy~\eqref{row-wise}. 

Let $A$ be an upper-triangular matrix with simple eigenvalues, thus diagonalizable. In the next theorem we prove that the Algorithm~\ref{agm:Jacobi} under the row-wise pivot strategy diagonalizes $A$ in only one cycle.

\begin{Theorem}\label{tm:triangular}
Let $A\in\C^{n\times n}$, $n\geq 3$, be an upper-triangular matrix with simple eigenvalues and let $A^{(N)}$, $N=n(n-1)/2$, be a matrix obtained from $A$ after one cycle of the Algorithm~\ref{agm:Jacobi} under the row-wise pivot strategy. Then, $A^{(N)}$ is a diagonal matrix.
\end{Theorem}

\begin{proof}
We first observe that one iteration~\eqref{Jacobi_iteration} on $A$ preserves the upper-triangular form. Assume that $(p,q)$, $p<q$, is a pivot position. Then, the pivot submatrix of $A$ is given by
$$\widehat{A}=\begin{bmatrix}
a_{pp} & a_{pq} \\
0 & a_{qq}
\end{bmatrix}.$$
Since $a_{qp}=0$, according to~\eqref{eq:xlin}, we have
\begin{equation}\label{eq:xtriangular}
x=\frac{a_{pq}}{a_{qq}-a_{pp}},
\end{equation}
and, according to~\eqref{eq:y}, $y=0$. From the relations~\eqref{elements}, for the obtained matrix $A'$ and $i\neq p,q$ we have 
$$a_{pp}'=a_{pp}, \quad a_{qq}'=a_{qq}, \quad a_{ip}'=a_{ip}, \quad a_{qi}'=a_{qi},$$
while 
$$a_{pi}'=a_{pi}-xa_{qi} \quad \text{and} \quad a_{iq}'=a_{iq}+xa_{ip}.$$
For $1\leq i<p$, we have $a_{pi}=a_{qi}=0$. Hence $a_{pi}'=0$. Similarly, for $q<i\leq n$, $a_{ip}=a_{iq}=0$, thus $a_{iq}'=0$. All other elements from the lower triangular are unchanged, so we establish that $A'$ is upper-triangular. 

Now we take the row-wise pivot strategy~\eqref{row-wise}. As we saw earlier, each iteration will keep the upper-triangular structure. Therefore, it is enough to check that the elements in the upper triangle, once they are annihilated, will stay zero. 

In the first iteration we get $a_{12}^{(1)}=0$. In the second iteration, we get $a_{13}^{(2)}=0$ and it follows from~\eqref{elements} that
$$a_{12}^{(2)}=a_{12}^{(1)}-x_1a_{32}^{(1)}=0,$$
since $a_{32}^{(1)}$ is an element from the lower triangle. Inductively, in the $j$th iteration, $1<j\leq n-1$, acting on the pivot position $(1,j+1)$, we get $a_{1,j+1}^{(j)}=0$, and, for $2\leq i\leq j$,
$$a_{1i}^{(j)}=a_{1i}^{(j-1)}-x_{j-1}a_{j+1,i}^{(j-1)}.$$ Then, because $a_{1i}^{(j-1)}=0$ and $(j+1,i)$ is a position in the lower triangle, it follows that $a_{1i}^{(j)}=0$. Therefore, after the first $n-1$ iterations, all off-diagonal elements in the first row are equal to zero. Following the same reasoning, we can conclude that the transformation annihilating the $(p_k,q_k)$-element, keeps the zeros on positions $(p_k,i)$, $p_k<i<q_k$, that is, to the left from $(p_k,q_k)$. 

Now, consider the $n$th iteration. It produces $a_{23}^{(n)}=0$ and we have
$$a_{12}^{(n)}=a_{12}^{(n-1)}=0$$
and
$$a_{13}^{(n)}=a_{13}^{(n-1)}+x_{n-1}a_{12}^{(n-1)}=0,$$
since both positions $(1,2)$ and $(1,3)$ are in the first row, which is already annihilated. In the same way it follows that all transformations acting on the second row keep the zeros from the first row. Inductively, transformation acting on position $(p_k,q_k)$ preserves the zeros in all upper rows $i$, $1\leq i<p_k$. 

In conclusion, under the row-wise strategy, once an element is annihilated, it remains zero until the end of the cycle, which implies that $A^{(N)}$ is diagonal. 
\end{proof}

Note that Theorem~\ref{tm:triangular} does not only hold for the row-wise pivot strategy. It also holds for several other strategies, e.g., for the column-wise pivot strategy, but not for an arbitrary cyclic pivot strategy.

We use Theorem~\ref{tm:triangular} to show that Algorithm~\ref{agm:Jacobi} converges on almost triangular matrices. Specifically, Theorem~\ref{tm:almosttriang} shows that, if $A$ is close enough to the upper-triangular form, matrix $A^{(N)}$ obtained after only one cycle of the naive Jacobi satisfies the conditions of Theorem~\ref{tm:conv}. Therefore, the naive Jacobi converges quadratically on $A^{(N)}$.

\begin{Theorem}\label{tm:almosttriang}
Let $A\in\C^{n\times n}$, $n\geq 3$ and  $A=B+L$, where $B$ is upper-triangular matrix with simple eigenvalues and $L$ is a strictly lower-triangular matrix. Let $A^{(N)}$, $N=n(n-1)/2$, be a matrix obtained from $A$ after one cycle of Algorithm~\ref{agm:Jacobi} under the row-wise pivot strategy. Set 
$$\mu_B\coloneqq\min_{i\neq j}|b_{ii}-b_{jj}|>0$$
and
$$\ell\coloneqq \max_{i>j}|a_{ij}|=\|L\|_{\max}.$$
Then there is a constant $c(B)$, depending only on $B$, such that 
\begin{equation}\label{eq:tmtriangassumption}
\ell<c(B)
\end{equation}
implies 
$$80(n-1)\frac{\eta_N}{\mu_N}\leq1.$$
That is, the matrix $A^{(N)}$ satisfies the conditions of Theorem~\ref{tm:conv}.
\end{Theorem}

\begin{proof}
We have 
$$A^{(N)}=V^{-1}AV=V^{-1}(B+L)V, \quad \text{where} \ V=T_0S_0T_1S_1\cdots  T_{N-1}S_{N-1}.$$
Because $B$ is upper-triangular, there is a transformation $W=T'_0S'_0T'_1S'_1\cdots  T'_{N-1}S'_{N-1}$ such that
$$W^{-1}BW=D$$
is a diagonal matrix. Hence, we can write
\begin{equation}\label{eq:D+E}
A^{(N)}=W^{-1}BW+V^{-1}(B+L)V-W^{-1}BW=D+E,
\end{equation}
where
$$E=V^{-1}(B+L)V-W^{-1}BW.$$

Take the function $\phi(B)\coloneqq B^{(N)}$ that represents one row-wise cycle of the Algorithm~\ref{agm:Jacobi} on the neighborhood of a triangular matrix $B$ with distinct diagonal elements. We claim that $\phi$ is continuous on $B$. One cycle is a finite composition of elementary pivot maps, so it is enough to check the continuity of one iteration. For $B$ and pivot pair $(p,q)$, the transformation parameter $x$ is calculated by the formula~\eqref{eq:xtriangular}. For a matrix close to $B$, we have
$$\begin{bmatrix}
a_{pp} & a_{pq} \\
\gamma & a_{qq}
\end{bmatrix},
\quad a_{pp}\neq a_{qq}.$$
Then, $x$ is determined by
$$\gamma x^2+(a_{qq}-a_{pp})x-a_{pq}=0.$$
Such $x$ depends continuously on the entries of the pivot submatrix and, when $\gamma\rightarrow0$, it tends to~\eqref{eq:xtriangular}. Thus, $\phi$ is continuous at $B$.

Therefore, for every $\epsilon>0$ there is $\delta_B=\delta_B(\epsilon)>0$ such that 
$$\|L\|_{\max}<\delta_B \ \Rightarrow \ \|\phi(B+L)-\phi(B)\|_{\max}=\|V^{-1}(B+L)V-W^{-1}BW\|_{\max}=\|E\|_{\max}<\epsilon.$$
In particular, for $C_n=80(n-1)$, there is $\delta_B>0$ such that 
\begin{equation}\label{eq:Emax}
\ell<\delta_B \ \Rightarrow \ \|E\|_{\max}<\frac{\mu_B}{C_n+2}.
\end{equation}
Therefore, $\delta_B$ depends only on $B$. Moreover, from the relation~\eqref{eq:D+E}, for $\eta_N$ defined as in~\eqref{eta_k}, we have \begin{equation}\label{eq:etaE}
\eta_N=\|E-\mathrm{diag}(E)\|_{\max}\leq\|E\|_{\max}\substack{\eqref{eq:Emax}\\ <}\frac{\mu_B}{C_n+2}.
\end{equation}
Since the diagonal entries of $D$ are equal to the diagonal entries of $B$, relation~\eqref{eq:D+E} also implies
$$a_{ii}^{(N)}-a_{jj}^{(N)}=b_{ii}-b_{jj}+e_{ii}-e_{jj}.$$
Then, by the reverse triangle inequality,
$$|a_{ii}^{(N)}-a_{jj}^{(N)}|\geq|b_{ii}-b_{jj}|-|e_{ii}|-|e_{jj}|.$$
Taking the minimum over $i\neq j$ and considering $\delta_B$ from~\eqref{eq:Emax} and $\mu_N$ as in~\eqref{mu_k}, we obtain
\begin{equation}\label{eq:Epom}
\mu_N\geq \mu_B-2\|E\|_{\max} \substack{\eqref{eq:Emax}\\ \geq} \mu_B-\frac{2\mu_B}{C_n+2}=\mu_B\frac{C_n}{C_n+2}.
\end{equation}
It follows that
$$C_n\frac{\eta_N}{\mu_N} \substack{\eqref{eq:Epom} \eqref{eq:etaE}\\ <} \frac{C_n+2}{\mu_B}\frac{\mu_B}{C_n+2}=1,$$
that is, 
\begin{equation}\label{eq:tmassumtionN}
80(n-1)\frac{\eta_N}{\mu_N}<1.
\end{equation}

Finally, we set
$c(B)=\delta_B.$
Then, the assumption~\eqref{eq:tmtriangassumption} implies $\ell<\delta_B$, that is, inequality~\eqref{eq:Emax} and, consequently, \eqref{eq:tmassumtionN} hold for $c(B)$.
\end{proof}

\section{Numerical examples without preconditioning}\label{sec:num_1}

In this section, we give several numerical examples where we test the Algorithm~\ref{agm:Jacobi}. The experiments were performed in MATLAB R2026a. The naive Jacobi algorithm is implemented under the row-wise pivot strategy with the stopping criterion
\begin{equation}\label{stop}
\text{if }\off\left(A^{(kN)}\right)<\mathrm{tol}, \text{then stop the algorithm.}
\end{equation}
We use $\mathrm{tol}=2.220\cdot 10^{-16}$,  MATLAB's double precision machine epsilon. The codes used to produce the results presented here, as well as in Section~\ref{sec:num_mp}, can be found at
https://github.com/Marija-Miloloza-Pandur/Naive-Jacobi.

\subsection{Matrices satisfying the quadratic convergence condition}

For our first numerical example, we generate the complex matrices $A$ such that the assumption~\eqref{eq:tmassumption} holds for $A$: 
\begin{verbatim}
    eta=0.01;
    mu=80*(n-1)*eta;
    phases=2*pi*rand(n,n);
    A=eta*exp(1i*phases);
    diag_elements=(1:n)*mu;
    A=A-diag(diag(A))+diag(diag_elements);
\end{verbatim}
In Figure~\ref{fig:example1} we see how the off-norm changes for different $n$, which is in line with Theorem~\ref{tm:conv}. The tested examples converged in three cycles. In Table~\ref{table:residual}  we report the computed relative residual $\|AV-V\Lambda\|_F/\|A\|_F$ in the Frobenius norm.

\begin{figure}[ht]
\includegraphics[width=.45\linewidth]{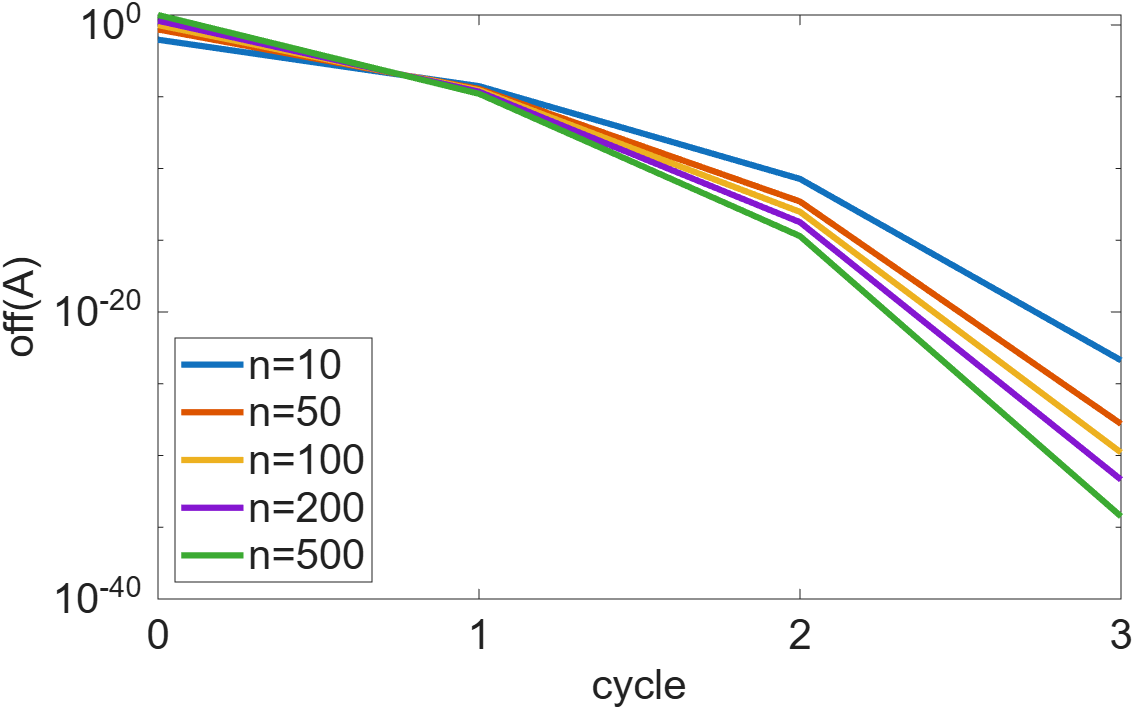}
\caption{Quadratic convergence of the naive Jacobi method}\label{fig:example1}
\end{figure}

\begin{table}[ht]
\centering
\begin{tabular}{cccccc}
\toprule
     & $n=10$ & $n=50$ & $n=100$ & $n=200$ & $n=500$ \\
\midrule
    $\|AV-V\Lambda\|_F/\|A\|_F$ & $1.1\cdot 10^{-16}$ & $1.7\cdot 10^{-16}$ & $2.1\cdot 10^{-16}$ & $3.4\cdot 10^{-16}$ & $3.8\cdot10^{-16}$ \\
    \bottomrule
\end{tabular}
\caption{The computed relative residual}\label{table:residual}
\end{table}

Furthermore, to assess the accuracy of the method,  we compare the naive Jacobi algorithm to the MATLAB \texttt{eig} function and the Eberlein algorithm~\cite{Eberlein62} with the stopping criterion~\eqref{stop}. For exact eigenvalues $\lambda_i^*$, $i=1,\ldots,n$, we take the eigenvalues computed to quadruple precision using the Advanpix Multiprecision Computing Toolbox~\cite{advanpix}. We compare the maximum and mean relative errors for the three methods for different matrix sizes $n$, $20\leq n\leq200$. In Figure~\ref{fig:example1acc} we see that the naive Jacobi outperforms the other two options by the order of magnitude.

\begin{figure}[ht]
\includegraphics[width=.45\linewidth]{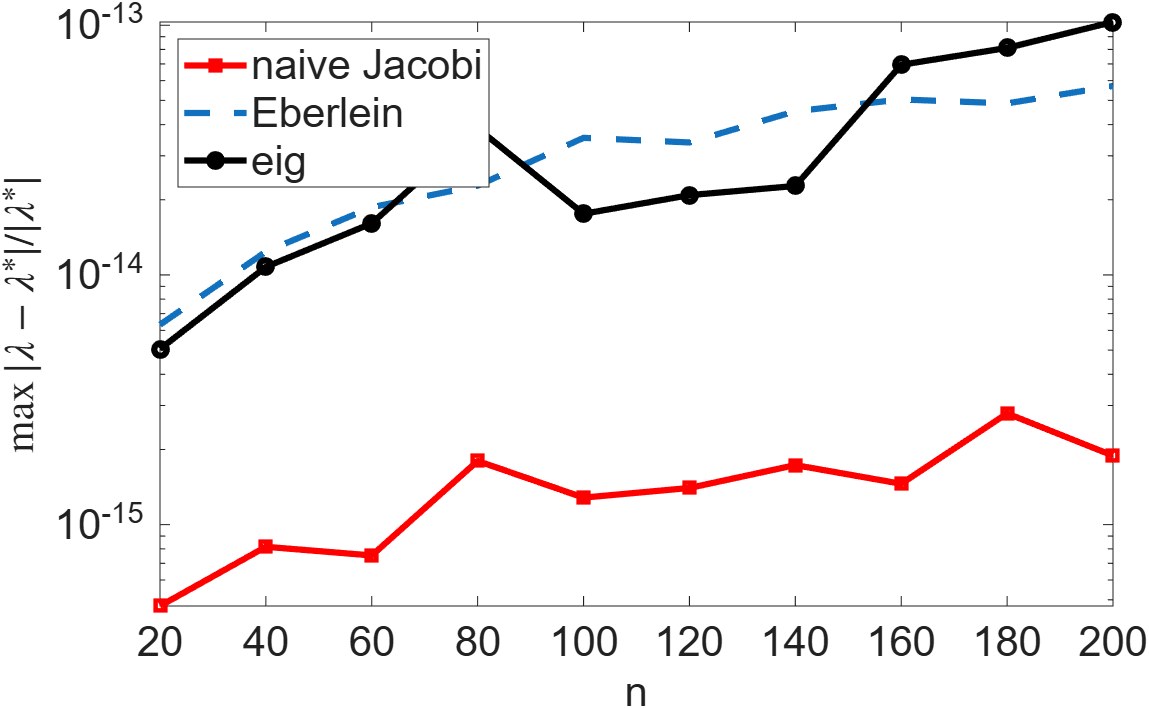}
\includegraphics[width=.45\linewidth]{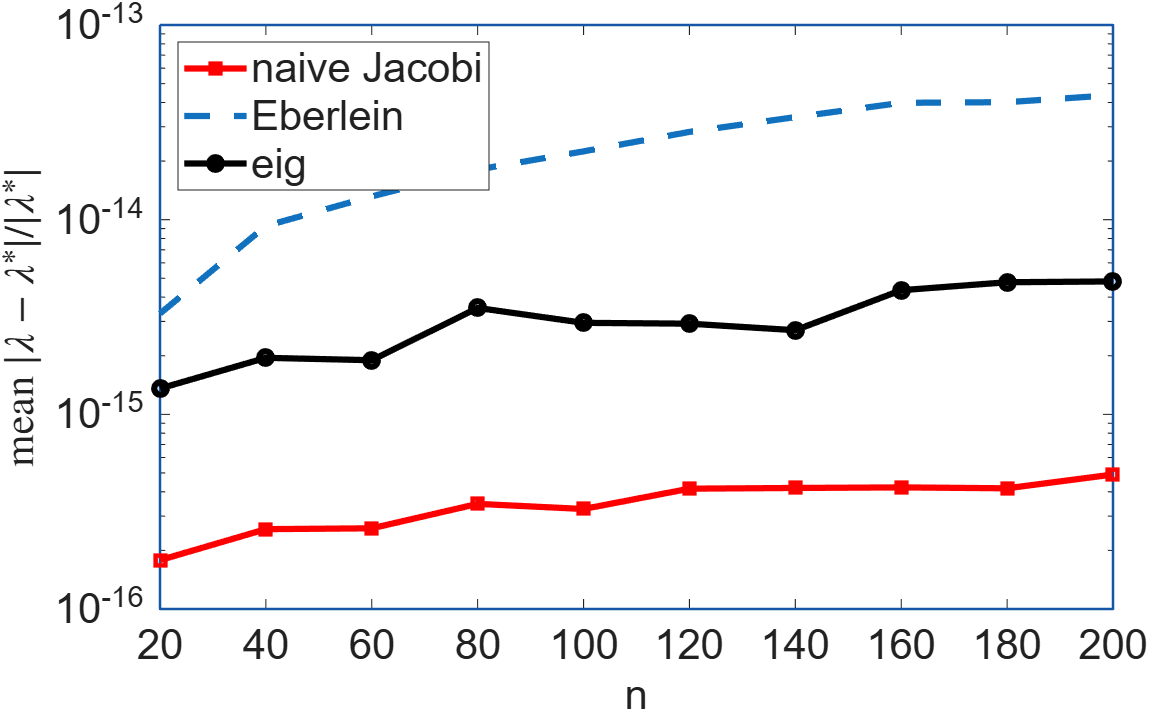}
\caption{Maximal (left)  and mean (right)  relative error of the naive Jacobi method}\label{fig:example1acc}
\end{figure}

Next, we form a Hermitian matrix $B=(A+A^*)/2$. The diagonal entries of $B$ are the same as the diagonal entries of $A$ and the largest off-diagonal element of $B$ is not larger than the largest off-diagonal element of $A$. Thus, assumption~\eqref{eq:tmassumption} holds for $B$. We apply on $B$ the naive Jacobi algorithm and compare it with the MATLAB \texttt{eig}, as well as the Jacobi algorithm for Hermitian matrices~\cite{BeHa21} with the stopping criterion~\eqref{stop}, for $20\leq n\leq 200$. The results are given in Figure~\ref{fig:example1accH}. Naive Jacobi gives the same results as the standard Jacobi, significantly more accurate than \texttt{eig}. 

\begin{figure}[ht]
\includegraphics[width=.45\linewidth]{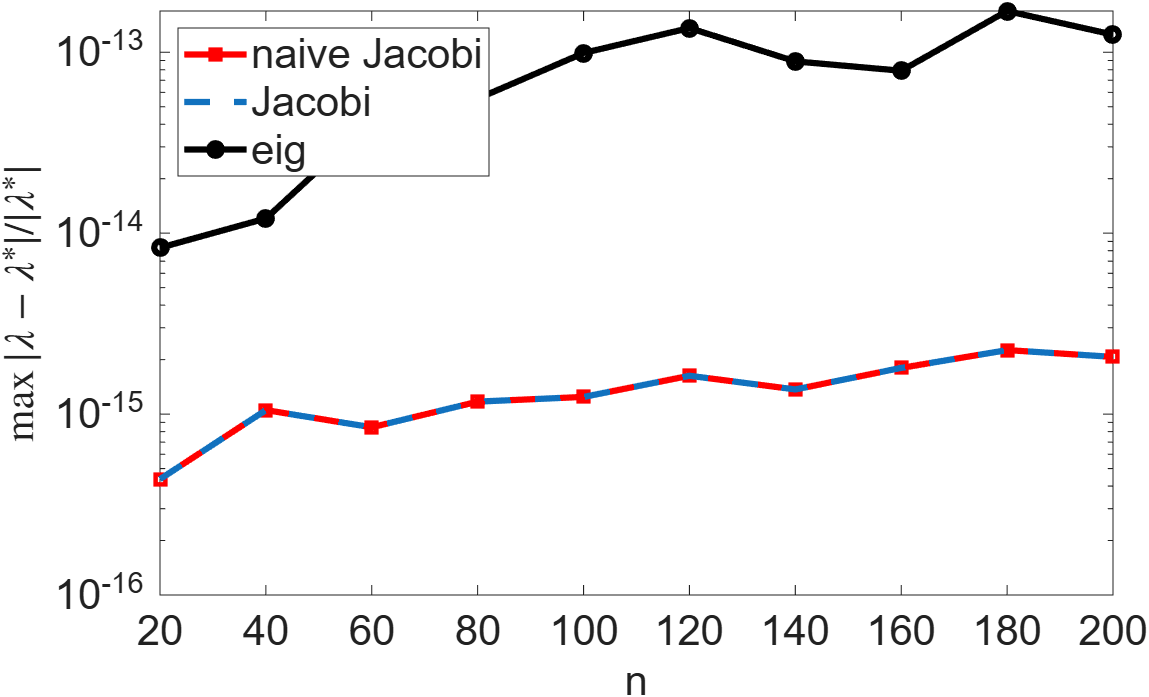}
\includegraphics[width=.45\linewidth]{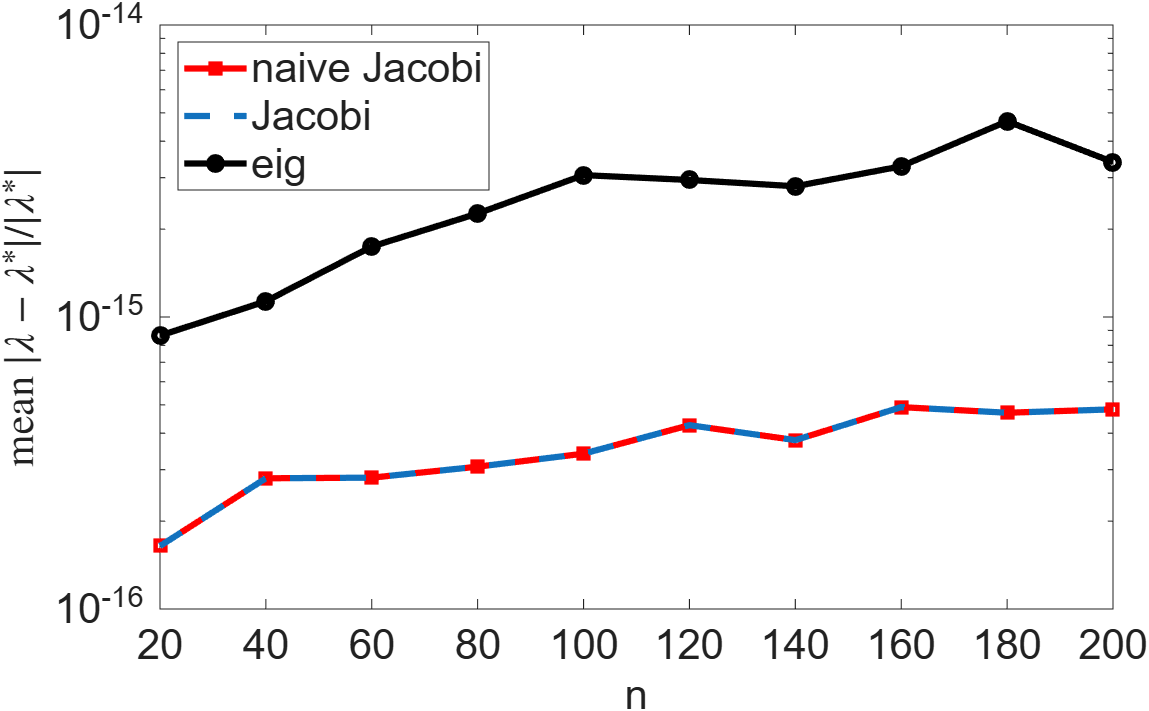}
\caption{Maximal (left) and mean (right)  relative error of the naive Jacobi method on Hermitian matrices}\label{fig:example1accH}
\end{figure}

Matrices satisfying the assumption~\eqref{eq:tmassumption} are very close to the diagonal form, which makes the conditions of Theorem~\ref{tm:conv} overly restrictive. However, as we are about to see in the rest of this section, the algorithm converges in many cases where~\eqref{eq:tmassumption} does not hold.

\subsection{$4\times4$ matrices}

We examine random complex $4\times4$ matrices. For $n=4$, assumption~\eqref{eq:tmassumption} holds if $\frac{\eta_0}{\mu_0}\leq\frac{1}{240}$. We ran the Algorithm~\ref{agm:Jacobi} on $1000$ matrices formed as
\begin{verbatim}
A=M*diag(v)*inv(M);
\end{verbatim}
where $v$ is a vector and $M$ is a matrix, both having complex values. This way we know the exact eigenvalues of $A$.
To make our results reproducible, we use the command $\texttt{rng(1)}$ to control the random number generator. The algorithm converged for all matrices, although the value $\frac{\eta_0}{\mu_0}$ was between $0.7$ and $51$, depending on a matrix, but never close to $\frac{1}{240}$. 

In addition to the convergence properties, we tested the high relative accuracy of the naive Jacobi method. We compared our results to those obtained by the MATLAB \texttt{eig} function. Specifically, we tested maximal relative error 
$$\text{err}_{\max}=\max_i\{|\lambda_i-\lambda_i^*|/|\lambda_i^*|\},$$
as well as the maximal error of the real and imaginary part,
\begin{align*}
\text{err}_{\max\text{Re}} & =\max_i\{|\text{Re(}\lambda_i)-\text{Re}(\lambda_i^*)|/|\text{Re}(\lambda_i^*)|\}, \\
\text{err}_{\max\text{Im}} & =\max_i\{|\text{Im(}\lambda_i)-\text{Im}(\lambda_i^*)|/|\text{Im}(\lambda_i^*)|\},
\end{align*}
where $\lambda_i$ is the computed and $\lambda_i^*$ is the exact eigenvalue.
In Table~\ref{table:4x4} we can see that the naive Jacobi exceeded \texttt{eig} in $81.4\%$ of the cases.

\begin{table}[ht]
\centering
\begin{tabular}{cccc}
\toprule
 & $\text{err}_{\max}$ & $\text{err}_{\max\text{Re}}$ & $\text{err}_{\max\text{Im}}$ \\
\midrule
naive Jacobi & \textbf{814} & \textbf{739} & \textbf{734} \\
\texttt{eig} & 186 & 261 & 266 \\
\bottomrule
\end{tabular}
\caption{How many times each algorithm outperformed the other one on a sample of $1000$ random complex $4\times4$ matrices.}\label{table:4x4}
\end{table}

Here it is important to notice that, although the algorithm converged for all tested random matrices, it does not really converge for all $4\times4$ matrices. Take
$$A=\begin{bmatrix}
0 & 1 & 0 & 0 \\ 
0 & 0 & 1 & 0 \\ 
0 & 0 & 0 & 1 \\ 
1 & 0 & 0 & 0 \\ 
\end{bmatrix}.$$
This is a diagonalizable matrix with simple eigenvalues. However, the Algorithm~\ref{agm:Jacobi} fails to diagonalize it since all nontrivial pivot submatrices are not diagonalizable. This issue can be solved by an initial preconditioning by a similarity transformation that perturbs the entries of $A$.

\subsection{$\alpha$-scaled diagonal dominant matrices}

In this section we recall the notion  of one specific type of almost diagonal matrices: scaled diagonally dominant matrices ~\cite{BarlowDemmel1990}. 
Let $A_S=\Omega+Z\in\C^{n\times n}$, where $\Omega$ is diagonal and $Z$  has a zero diagonal ($Z$ is the off-diagonal part of $A_S$) and $|\Omega_{ii}|=1$, $i=1,\ldots,n.$ Let $D_1$ and $D_2$ be arbitrary nonsingular diagonal matrices of order~$n$. A~complex square matrix  $A=D_1A_SD_2$ \textbf{is called $\alpha$-scaled diagonally dominant with respect to a given matrix norm} $\vert\vert\cdot\vert\vert$ if  $$ \off_{\vert\vert}(A_S):=\vert\vert Z\vert\vert\leq \alpha, \quad 0\leq \alpha<1.$$ Note that an $\alpha$-scaled diagonally dominant matrix has nonzero diagonal elements.

Let~$A\in\C^{n\times n}$ be a complex matrix with the non-zero diagonal elements, then we can define
\begin{equation}\label{eq: As}
    A_S:=D^{-1}AD^{-1},\quad D =|diag(A)|^{1/2},
\end{equation}
where the modulus and square root are taken entrywise on the diagonal. Therefore, $|(A_{S})_{ii}|= 1$ for all~$i.$ If $\off_{\vert\vert}(A_S)\leq \alpha<1$, then $A=DA_SD$ is $\alpha$-scaled diagonally dominant. Since, $\alpha<1$, a matrix $A_S$, and consequently $A$ must be nonsingular (when $\vert\vert\cdot\vert\vert$ is induced  norm). 

 In addition to the off-norm $\off(A^{(kN)})$, we also monitor the scaled off-norm $\off(A_S^{(kN)})$ for $A_S$ from~\eqref{eq: As}, after each  cycle in our example. 
  Naive Jacobi may converge on $\alpha$-scaled diagonally dominant matrices with several equal diagonal elements, even   in the case with multiple eigenvalues (see the second type of matrices in the next example). We note that the naive Jacobi method does not converge for all $\alpha$-scaled diagonally dominant matrices.


In this example we form $\alpha$-scaled diagonally dominant matrices  for several different values of~$\alpha$, with the following MATLAB code:
\begin{verbatim}
    A = 100*randn(n)+100*1i*randn(n);
    A_off = A - diag(diag(A));
    off_norm = norm(A_off, 'fro');
    if off_norm >= alpha
        A_off = (alpha / (off_norm + eps)) * A_off;
    end
    d=randn(n,1)+randn(n,1)*1i;
    d=d./abs(d);
    As = A_off + diag(d);
\end{verbatim}  

For the first type we use  $n=100$ and $D_{jj}=5$, for $j=1,\ldots,n$, and form $A=DA_SD.$ 
The condition number of $A$ is from $1.03$ to $1.27$ for $\alpha=0.1,0.3,0.5,0.7,0.9.$
For example, for $\alpha=0.5$, the orders of magnitude of the off-norm $\off\left(A^{(kN)}\right)$ and the scaled off-norm $\off\left(A_S^{(kN)}\right)$ after each cycle (with the first value corresponding to the initial matrix $A=A^{(0)}$) are as follows:
\begin{center}
\begin{tabular}{cr}
  off-norm & $10^1, 10^{0},10^{-1}, 10^{-3}, \mathbf{10^{-8}, 10^{-17}}$ \\
scaled off-norm  & $10^{-1}, 10^{-1},10^{-2}, 10^{-5}, 10^{-9}, 10^{-18}.$ \\ 
\end{tabular}
\end{center}
  The relation~\eqref{eq:tmassumption} holds after the fourth cycle, and after that, we have the quadratic reduction of the off-norm with the constant $C(n,\mu_{4N})=\frac{80(n-1)}{\mu_{4N}}\sqrt{2N}\approx 6.1\cdot 10^{6}$, for $N=n(n-1)/2$.

For the second type, we form a matrix $D$ of order $70$ having diagonal elements from $10^{-3}$ to $10^3$, and then form $A$ as $DA_SD$. To have multiple eigenvalues,  we form the block-diagonal  matrix $B=\mathrm{diag}(A,A,A)$, for the matrix $A$ 
and apply the naive Jacobi algorithm to the matrix $B$ of order~$210$. The condition number of $B$ is of order $10^{12}$.

 Figure~\ref{fig: asdd diff alpha}
 contains the off-norm and the scaled off-norm after each cycle: for the first, the second type,  the norms are in subfigures (A), (B), respectively. The norms are monotonically decreasing, and for the first type  exhibit very similar behavior. The algorithm has converged within 3--6 cycles in all cases. As expected, the algorithm converges the fastest for the smallest $\alpha=0.1$.

\begin{figure}[ht]
\centering
\begin{subfigure}{0.45\textwidth}
\includegraphics[width=\textwidth]{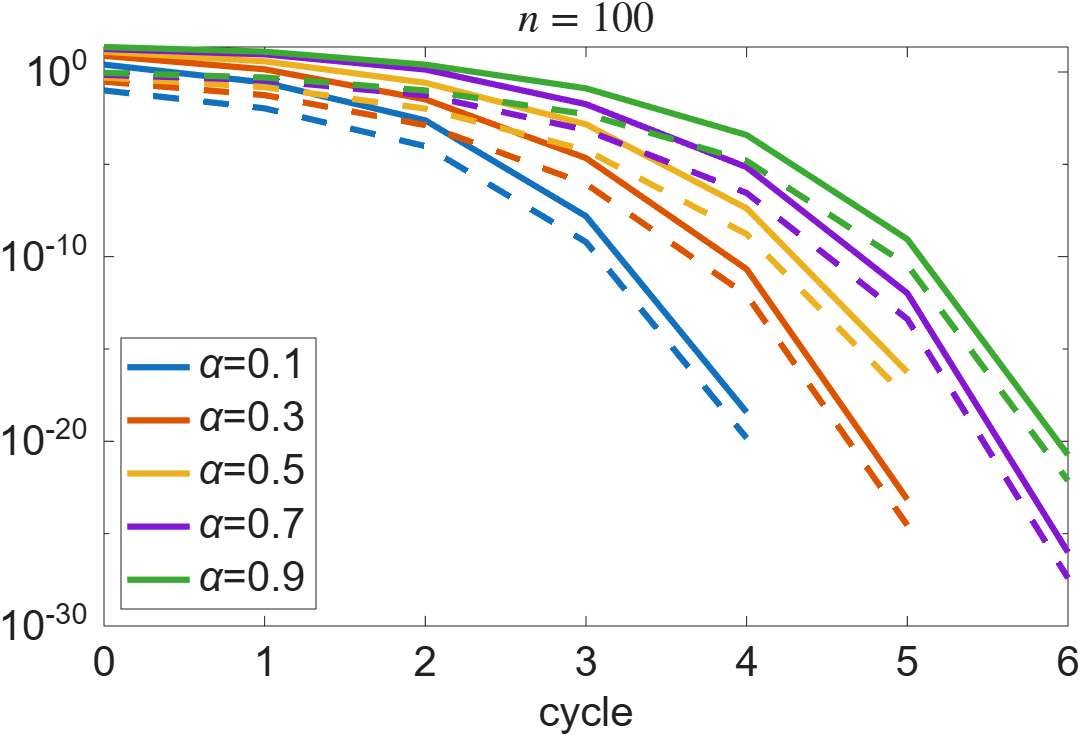}
\caption{Simple eigenvalues, $D_{ii}=5$}
\end{subfigure}
\begin{subfigure}{0.45\textwidth}
\includegraphics[width=\textwidth]{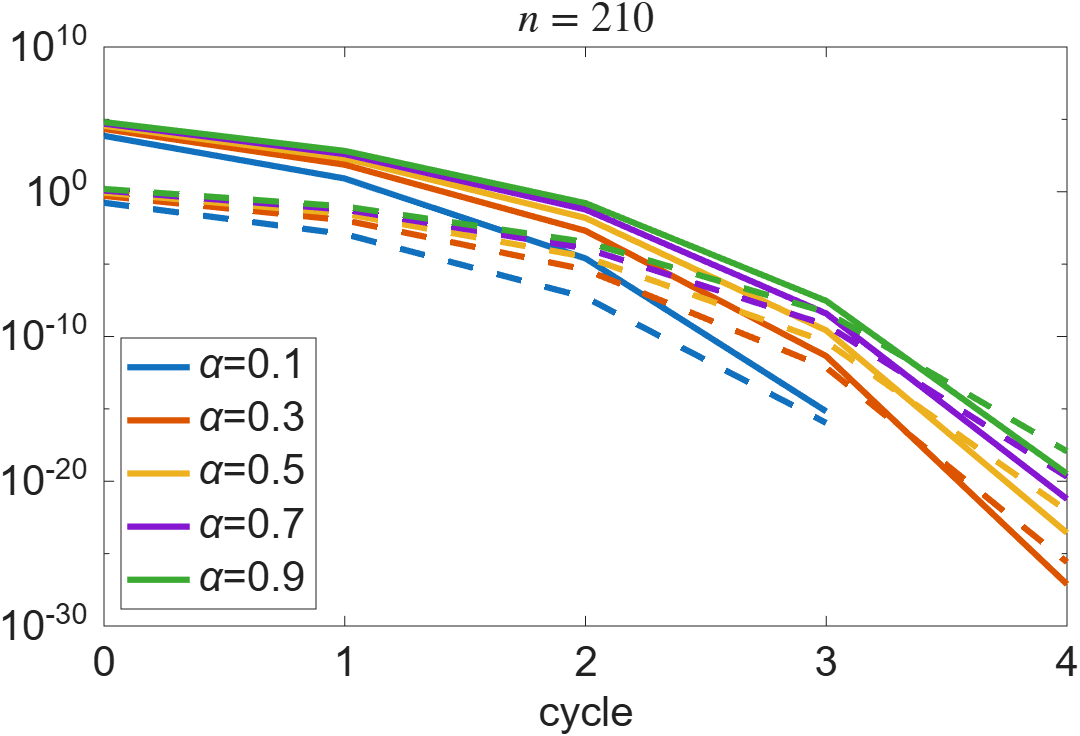}
\caption{Multiple eigenvalues, $D_{ii}\in [10^{-3},10^3]$} 
\end{subfigure}
\caption{Convergence of the off-norm (solid line) and the scaled off-norm (dashed line) on $\alpha$-scaled diagonally dominant matrices}\label{fig: asdd diff alpha}
\end{figure}
In  Table~\ref{table:asdd}, we give the computed relative residuals. For simple eigenvalues, the computed relative residual is of order $10^{-16}$, the double precision machine epsilon.  For multiple eigenvalues and $\alpha\neq 0.1$, the residual is even smaller.

\begin{table}[ht]
\centering
\begin{tabular}{cccccc}
\toprule
     & $\alpha=0.1$ & $\alpha=0.3$ & $\alpha=0.5$ & $\alpha=0.7$ & $\alpha=0.9$\\
\midrule
    simple evs & $3.8\cdot 10^{-16}$ & $4.0\cdot 10^{-16}$ & $4.6\cdot 10^{-16}$ & $4.7\cdot 10^{-16}$ & $6.3\cdot 10^{-16}$ \\
    multiple evs & $1.3\cdot 10^{-16}$ & $6.6\cdot 10^{-17}$ & $5.5\cdot 10^{-17}$ & $7.1\cdot 10^{-17}$ & $7.3\cdot 10^{-17}$ \\
    \bottomrule
\end{tabular}
\caption{The computed relative residual for $\alpha$-scaled diagonally dominant matrices }\label{table:asdd}
\end{table}

\section{Mixed-precision algorithm}\label{sec:mp}

We saw in Section~\ref{sec:cvg} that the naive Jacobi algorithm converges on almost diagonal matrices, as well as on almost triangular matrices, with simple eigenvalues. In order to extend the practical applicability of the method to a much wider class of matrices with simple eigenvalues, we employ a preconditioning step that transforms $A$ into a suitable form
\begin{equation}\label{eq:precond}
\widetilde{A}=\widetilde{V}^{-1}A\widetilde{V}.
\end{equation}
Here, $\widetilde{V}$ is a transformation that brings $A$ closer to the form assumed by the Theorem~\ref{tm:conv} or~\ref{tm:almosttriang}. In an ideal case, $\widetilde{A}$ would satisfy the conditions of Theorem~\ref{tm:conv} or~\ref{tm:almosttriang}, but, as we will see later, it is not necessary in practice. After preconditioning, we apply Algorithm~\ref{agm:Jacobi} to $\widetilde{A}$ and get its eigenvalues, which are the same (in exact arithmetic) as the eigenvalues of $A$, along with the eigenvectors, which are obtained by taking the product of $\widetilde{V}$ and the matrix obtained by the naive Jacobi algorithm.

To make the computing process more efficient, we compute the preconditioner in a lower precision, single or half. With appropriate hardware, using lower precision unit round-off $\mathbf{u_l}$, $\mathbf{u_l}\geq \mathbf{u}$, can greatly reduce computational time. Then we apply matrix similarity~\eqref{eq:precond} either at working precision $\mathbf{u}$, or at a higher precision $0<\mathbf{u_h}\leq \mathbf{u}$. Our mixed-precision implementation adopts the following steps.
\begin{itemize}
\item[(i)] Single or half precision: Find $V_l$ such that ${V}_l^{-1}A_lV_l$, for $A_l=\operatorname{fl}_{32}(A)$ or $A_l=\operatorname{fl}_{16}(A)$, is diagonal or upper-triangular at low precision.
\item[(ii)] Double or quadruple precision: Compute $\widetilde{A}=\widetilde{V}^{-1}A\widetilde{V}$, for $\widetilde{V}=\operatorname{fl}_{64}(V_l)$ or $\widetilde{V}=\operatorname{fl}_{128}(V_l)$.
\item[(iii)] Double precision: Apply Algorithm~\ref{agm:Jacobi} to $\operatorname{fl}_{64}(\widetilde{A})$.
\end{itemize}

For the first step, we have two options, depending if we rely on Theorem~\ref{tm:conv} or~\ref{tm:almosttriang}. If we want to move $A$ closer to the diagonal form, we  use the MATLAB $\texttt{eig}$ function at precision $\mathbf{u_l}$. If we aim at the triangular form we use either the Schur decomposition computed by the MATLAB $\texttt{schur}$ function at precision $\mathbf{u_l}$,
or several iterations of the QR eigenvalue algorithm. 

\begin{algorithm}[ht]
\caption{Unitary preconditioned naive Jacobi algorithm (two precisions)}\label{agm:nJmp}
\renewcommand{\algorithmicrequire}{\textbf{Input:}}
\renewcommand{\algorithmicensure}{\textbf{Output:}}
\begin{algorithmic}
\Require ${A}\in\C^{n\times n}$ with simple eigenvalues, low and working precisions $\mathbf{u_l}$ and $\mathbf{u}$, $0<\mathbf{u}\leq\mathbf{u_l}$.
\Ensure An approximate spectral decomposition $A=V\Lambda V^{-1}$.
\State \textbf{1:} Find $V_l$ such that $V_l^{-1}AV_l$ is upper-triangular at precision  $\mathbf{u_l}$. \Comment{\texttt{schur} or QR} 
\State \textbf{2:} Set $\widetilde{A}=\widetilde{V}^{-1}A \widetilde{V}$ at precision $\mathbf{u}$.
\State \textbf{3:} Compute $\widetilde{A}=W\Lambda W^{-1}$ using Algorithm~\ref{agm:Jacobi} at precision $\mathbf{u}$.
\State \textbf{4:} Set $V= \widetilde{V}W$ at precision $\mathbf{u}$.
\end{algorithmic}
\end{algorithm} 

\begin{algorithm}[ht]
\caption{\texttt{eig}-preconditioned naive Jacobi algorithm (three precisions)}\label{agm:nJmp3eig}
\renewcommand{\algorithmicrequire}{\textbf{Input:}}
\renewcommand{\algorithmicensure}{\textbf{Output:}}
\begin{algorithmic}
\Require ${A}\in\C^{n\times n}$ with simple eigenvalues, low, working, and high  precisions $\mathbf{u_l}$, $\mathbf{u}$, $\mathbf{u_h}$, $0<\mathbf{u_h}\leq \mathbf{u}\leq \mathbf{u_l}$.
\Ensure An approximate spectral decomposition $A=V\Lambda V^{-1}$.
\State \textbf{1:} Find $V_l$ such that $V_l^{-1}AV_l$ is diagonal at precision $\mathbf{u_l}$. \Comment{\texttt{eig}}
\State \textbf{2:} Set $\widetilde{A}=\widetilde{V}^{-1}A \widetilde{V}$ at precision $\mathbf{u_h}$.
\State \textbf{3:} Compute $\widetilde{A}=W\Lambda W^{-1}$ using Algorithm~\ref{agm:Jacobi} at precision $\mathbf{u}$.
\State \textbf{4:} Set $V= \widetilde{V}W$ at precision $\mathbf{u}$.
\end{algorithmic}
\end{algorithm} 

If the Schur decomposition or iterations of the QR algorithm are used in step (i) of the Algorithm~\ref{agm:nJmp}, matrix $V_l$ is numerically unitary at precision $\mathbf{u_l}$. When converted to higher precision, $\widetilde{V}$ is not unitary, but is close to a unitary matrix, and its condition number is close to one. That makes $\widetilde{V}$ a stable preconditioner and we perform step (ii) at double precision. On the other hand, if the eigenvalue decomposition is used in (i), then, in a general case, $V_l$ is not unitary at precision $\mathbf{u_l}$. Thus, its condition number can be higher. In that case, we have observed that accuracy of the algorithm improves slightly more if step (ii) is done at quadruple precision.

We summarize the first case in Algorithm~\ref{agm:nJmp} that uses two precisions, and the second case in Algorithm~\ref{agm:nJmp3eig} with three precisions. In Algorithm~\ref{agm:nJmp} we assume the row-wise pivot strategy, while Algorithm~\ref{agm:nJmp3eig} can be used under any cyclic pivot strategy. We provide numerical experiments in the next section.

The following Propositions assure the convergence of Algorithm~\ref{agm:nJmp} and~\ref{agm:nJmp3eig}.

\begin{Proposition}\label{prop:mptriangular}
Let $A\in\mathbb C^{n\times n}$, $n\ge3$, be an input matrix of Algorithm~\ref{agm:nJmp}. After the preconditioning step, let
$$A^{(0)}=\widetilde A=B+L,$$
where $B$ is upper triangular with simple eigenvalues and $L$ is strictly lower triangular. Set
$$\ell=\|L\|_{\max}.$$
If $\ell<c(B)$, where $c(B)$ is the constant from the Theorem~\ref{tm:almosttriang}, then, after one cycle of the naive Jacobi algorithm, inequality
$$80(n-1)\frac{\eta_N}{\mu_N}\leq 1,$$
holds, for $\mu_N$ and $\eta_N$ defined  in~\eqref{mu_k} and~\eqref{eta_k}, respectively. Consequently, the naive Jacobi algorithm applied to $A^{(N)}$ converges quadratically.
\end{Proposition}

\begin{proof}
Since
$$
A^{(0)}=\widetilde A=B+L,
$$
with $B$ upper triangular and $L$ strictly lower triangular,
all assumptions of Theorem~\ref{tm:almosttriang}
are satisfied.
The assumption
$$\ell<c(B)$$
implies
$$80(n-1)\frac{\eta_N}{\mu_N}\le1,$$
where $\eta_N$ and $\mu_N$ denote the quantities after one cycle of the
naive Jacobi algorithm.

Therefore, the matrix $A^{(N)}$ satisfies the assumptions of
Theorem~\ref{tm:conv}$,$ which yields quadratic convergence from the
second cycle onward.
\end{proof}

In the case of \texttt{eig} preconditioner (Algorithm~\ref{agm:nJmp3eig}), the off-norm of the matrix $\widetilde{A}$, the starting matrix of the naive Jacobi algorithm, will be small, of order $\off(A)\,c\,\mathbf{u_l}$ for some $c>0$.

\begin{Proposition}\label{prop:mp}
Let $A\in\C^{n\times n}$, $n\geq3$, be an input matrix of Algorithm~\ref{agm:nJmp3eig}. After the preconditioning step, let $A^{(0)}=\widetilde{A}$, and suppose that 
\begin{equation}\label{eq:offAtilde}
    \off(A^{(0)})\leq \varepsilon, \quad \varepsilon>0,
\end{equation}
Let $\lambda_1^{(0)},\ldots,\lambda_n^{(0)}$ be the eigenvalues of
$A^{(0)}$, and denote their spectral gap by 
$$\delta^{(0)}=\min_{i\neq j}|\lambda_i^{(0)}-\lambda_j^{(0)}|>0.$$
If
\begin{equation}\label{eq:conv0}
    (80(n-1)+2)\varepsilon\leq\delta^{(0)},
\end{equation}
then inequality
$$80(n-1)\frac{\eta_0}{\mu_0}\leq1$$
holds, for $\mu_0$ and $\eta_0$ defined in~\eqref{mu_k} and~\eqref{eta_k}, respectively.
Consequently, the naive Jacobi algorithm applied to $\widetilde{A}$ converges quadratically.
\end{Proposition}

\begin{proof}
By assumption~\ref{eq:offAtilde}, we have
\begin{equation}\label{eq:eta_0}
    \eta_0
    =\max_{i\neq j}|a^{(0)}_{ij}|
    \leq\off (A^{(0)})
    \leq\varepsilon.
\end{equation}

Let $D=\mathrm{diag}\left(a_{11}^{(0)},a_{22}^{(0)},\ldots,a_{nn}^{(0)}\right)$ and $E=A^{(0)}-D$. 
Then,
$$\|E\|_2\leq\|E\|_F=\|A^{(0)}-D\|_F=\off(A^{(0)})\leq\varepsilon.$$
Matrices $D+E$ and $A^{(0)}$ have the same eigenvalues. Since  $D$ is normal, the  Bauer--Fike theorem implies that  for every eigenvalue $\lambda_i^{(0)}$ of $A^{(0)}$,
there exists an index $j(i)$ such that
$$|a_{j(i)j(i)}^{(0)}-\lambda_i^{(0)}| \leq \|E\|_2 \leq \varepsilon.$$
Condition~\eqref{eq:conv0} implies 
$\delta^{(0)}>2\varepsilon$. Hence, the eigenvalues of $A^{(0)}$ are separated by more than $2\varepsilon$, and therefore two different eigenvalues  cannot correspond to the same diagonal entry of $D$. Hence, $i\mapsto j(i)$ is a  permutation and we may assume that 
$$|a^{(0)}_{ii}-\lambda_i^{(0)}|\leq \varepsilon,\quad i=1,\ldots,n.$$

Now, for $i\neq j$, we get
$$|a^{(0)}_{ii}-a^{(0)}_{jj}| \geq |\lambda_i^{(0)}-\lambda_j^{(0)}| -|a^{(0)}_{ii}-\lambda_i^{(0)}| -|a^{(0)}_{jj}-\lambda_j^{(0)}| \geq |\lambda_i^{(0)}-\lambda_j^{(0)}|-2\varepsilon.$$
Taking the minimum over $i\neq j$ yields
$$\mu_0\geq\delta^{(0)}-2\varepsilon,$$
and, using ~\eqref{eq:conv0} and~\eqref{eq:eta_0}, we obtain
$$  \mu_0\geq80(n-1)\varepsilon\geq80(n-1)\eta_0,$$
  which proves the claim.
\end{proof}


\section{Numerical examples with mixed precision preconditioning}\label{sec:num_mp}

In this section we discuss numerical results attained by the Algorithms~\ref{agm:nJmp} and~\ref{agm:nJmp3eig}. The Advanpix Multiprecision Computing Toolbox~\cite{advanpix} is used for multi-precision computations.

\subsection{Assessing different preconditioners}

We start this section by assessing how much di\-ffe\-rent preconditioners reduce the matrix off-norm, for the \texttt{eig} preconditioner (Algorithm~\ref{agm:nJmp3eig}), that is, the Frobenius norm of the lower triangle, denoted by $\off_L$, for the Schur preconditioner (Algorithm~\ref{agm:nJmp}). For each $n$, $20\leq n\leq200$, we used five random matrices and plotted the mean value of $\off(\widetilde{A})/\off(A)$ and $\off_L(\widetilde{A})/\off_L(A).$
The results are visible in Figure~\ref{fig:preconditioner}. As expected, when the precision is lower, the reduction is lower. It can be observed from the left-hand side in Figure~\ref{fig:preconditioner} that quarter-precision preconditioner is not beneficial for the bigger matrices.

\begin{figure}[ht]
\includegraphics[width=.45\linewidth]{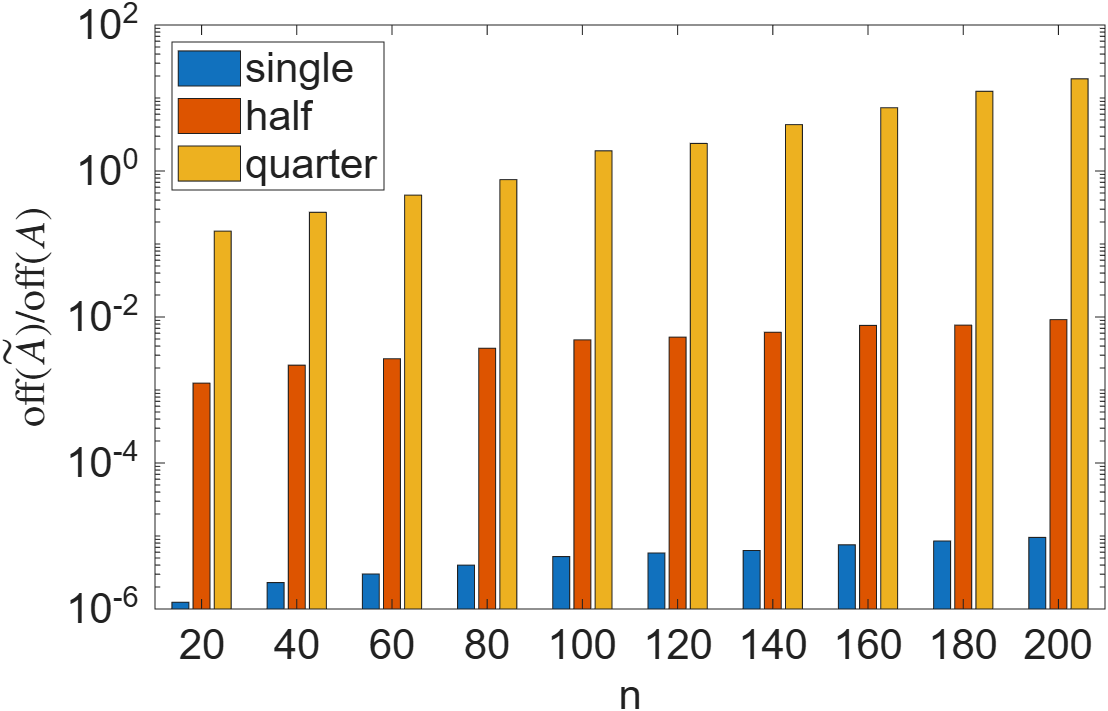} \quad \includegraphics[width=.45\linewidth]{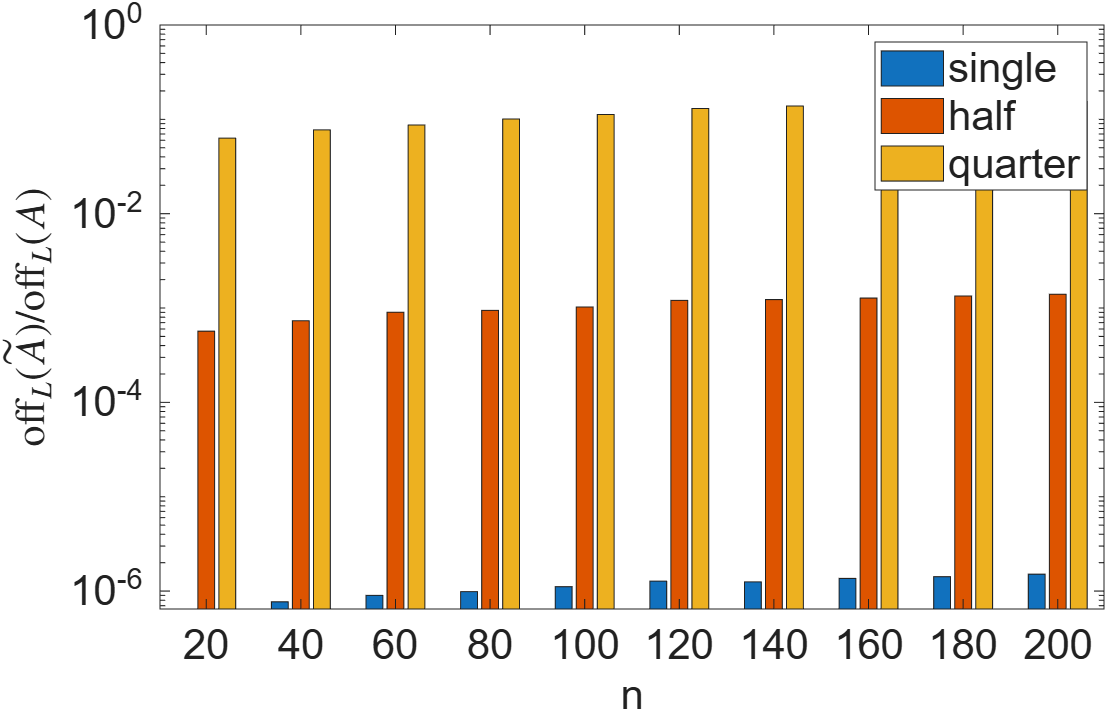}
\caption{$\off$ reduction for the \texttt{eig} preconditioner (left) and $\off_L$ reduction for the Schur preconditioner (right)}\label{fig:preconditioner}
\end{figure}

In Figure~\ref{fig:mp_cvg} we show the convergence of the preconditioned matrix for both preconditioners, for $n=200$ and $n=400$, on random complex matrices. We used the command $\texttt{rng(1)}$ to ensure reproducibility. In Table~\ref{table:cvg} we report the values $80(n-1)\eta_0/\mu_0$, where $\eta_0$ and $\mu_0$ are parameters of the preconditioned matrix $\widetilde{A}$. In order to fulfill the convergence criterion~\eqref{eq:tmassumption}, this value should be at most one. As it can be seen in the table, the condition did not hold. Nevertheless, the naive Jacobi algorithm converged.

\begin{figure}[ht]
\includegraphics[width=.45\linewidth]{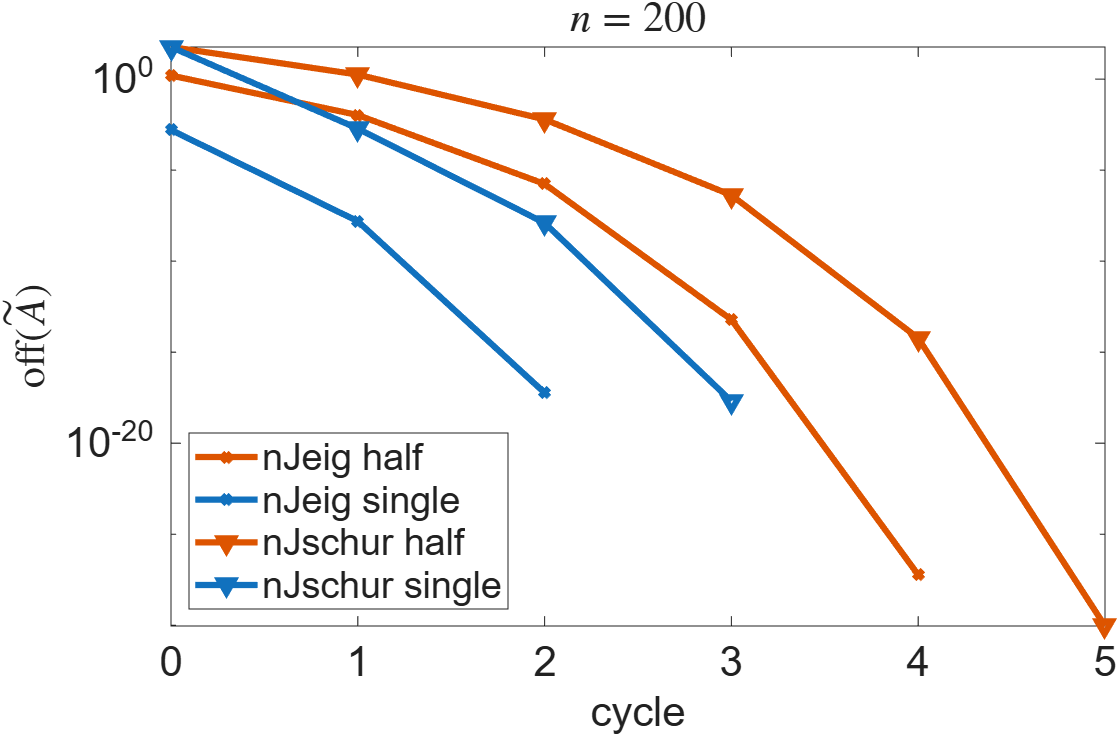} \quad \includegraphics[width=.45\linewidth]{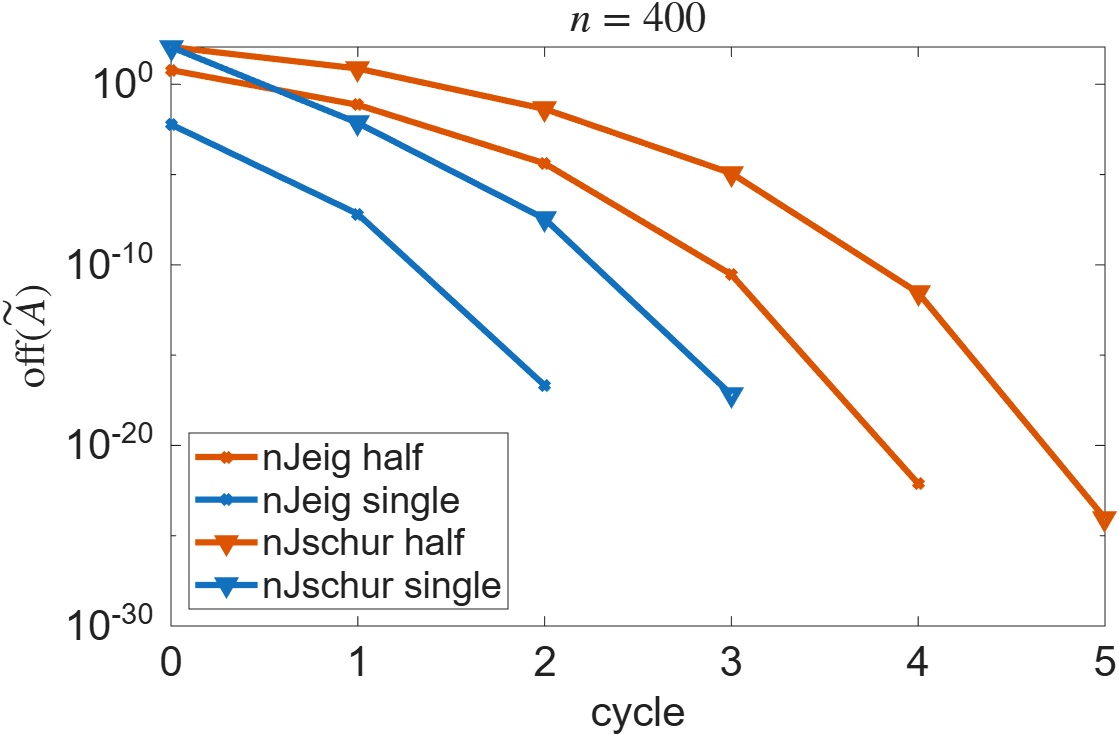}
\caption{Convergence of the mixed-precision naive Jacobi with the \texttt{eig} and Schur preconditioners}\label{fig:mp_cvg}
\end{figure}

\begin{table}[ht]
\centering
\begin{tabular}{ccccc}
\toprule
 & nJeig half & nJeig single & nJschur half & nJschur single \\
\midrule
$n=200$ & 13965 & 8 & 134170 & 127400 \\
$n=400$ & 162836 & 118 & 414303 & 361359 \\
\bottomrule
\end{tabular}
\caption{Values $80(n-1)\eta_0/\mu_0$ for the preconditioned matrices}\label{table:cvg}
\end{table}

\begin{figure}[ht]
\includegraphics[width=.45\linewidth]{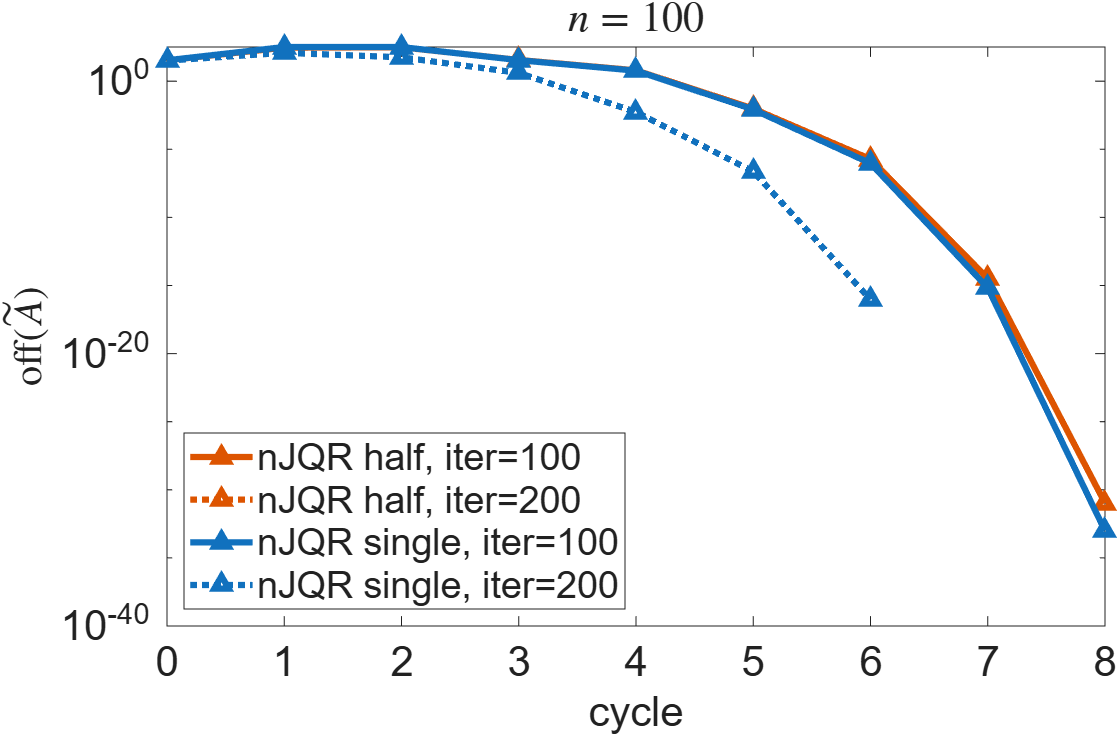} \quad \includegraphics[width=.45\linewidth]{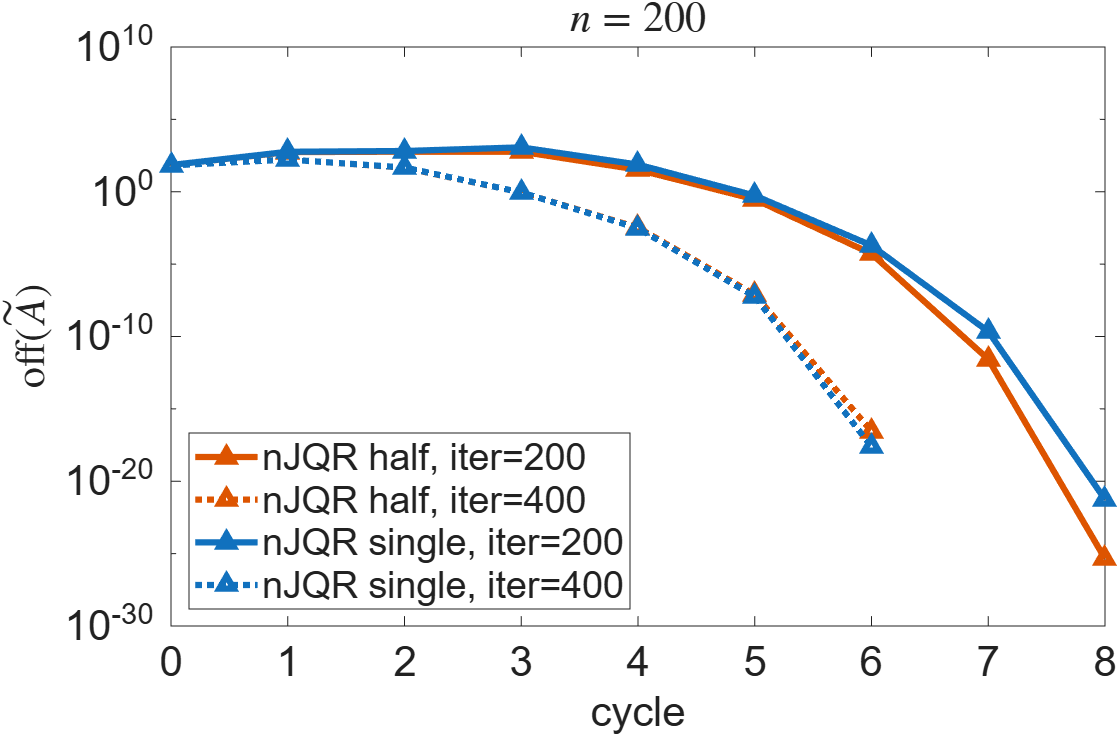}
\caption{Convergence of the mixed-precision naive Jacobi with the QR preconditioner}\label{fig:mp_cvgQR}
\end{figure}

For the QR preconditioner, we have observed that in practice, it is not necessary to do the full QR algorithm until convergence. This is useful because the QR algorithm can be time-consuming. In practice, it seams that the number of QR iterations needed for the convergence of the naive Jacobi is up to $n$ for $n=1000$. For a larger $n$, this number may increase. In Figure~\ref{fig:mp_cvgQR} we present the convergence behavior of the naive Jacobi algorithm on the preconditioned matrix, after different numbers of QR iterations. We did the experiment for $n=100$ and $n=200$ on random complex matrices generated with the command $\texttt{rng(1)}$. We used the QR algorithm with shifts. 

\begin{table}[ht]
\centering
\begin{tabular}{cccccc}
\toprule
& $n=20$ & $n=40$ & $n=60$ & $n=80$ & $n=100$  \\
\midrule
nJeig half & 13.5098 & 30.6013 & 47.5730 & 47.1365 & 63.8975  \\
nJeig single & 13.5090 & 30.5361 & 47.4275 & 47.0373 & 63.8406  \\
nJschur half &  1.0017 & 1.0029 & 1.0036 & 1.0049 & 1.0055  \\
nJschur single & 1.0000 & 1.0000 & 1.0000 & 1.0000 & 1.0000  \\
\bottomrule
\end{tabular}
\caption{Condition number of the preconditioner}\label{table:cond}
\end{table}

Moreover, we computed the condition numbers of the preconditioners (after they are transformed to double precision for the Algorithm~\ref{agm:nJmp}, that is, to quadruple precision for the Algorithm~\ref{agm:nJmp3eig}). The experiment is done on random matrices using fixed random seed \texttt{rng(1)}. In Table~\ref{table:cond} we give the results for the preconditioners attained in half or single precision, for $20\leq n\leq100$. As it was said in Section~\ref{sec:mp}, for the Schur approach, preconditioning matrix $\widetilde{V}$ remains nearly orthogonal, thus, $\text{cond}(\widetilde{V})$ is close to one. The condition numbers in the case of the QR preconditioner are the same, since $V_l$ is orthogonal in half/single precision, so we do not list them in the table. When using \texttt{eig} preconditioner, $\widetilde{V}$ may have large condition number, since $V_l$ is not orthogonal in the low precision, either.

\subsection{Accuracy of the mixed-precision naive Jacobi}

Next, we test the maximal forward error of the preconditioned naive Jacobi method compared to the MATLAB \texttt{eig} function. For the exact eigenvalues, we take the eigenvalues computed with \texttt{eig} in quadruple precision. The eigenvalues are paired using a minimum-cost assignment. We present the maximal relative errors.

\subsubsection{Varying matrix size}

For the Figure~\ref{fig:mp_acceig} we used \texttt{eig} preconditioner in single and half precision. We conducted two sets of experiments, one on small matrices, $10\leq n\leq100$, and the other on larger matrices, $50\leq n\leq500$. For each $n$ we tested five random complex matrices and plotted the average case. As it can be observed from the figure, the preconditioned naive Jacobi always gave more accurate eigenvalues than MATLAB's \texttt{eig}, better for approximately one order of magnitude. 
In the same way, in Figure~\ref{fig:mp_accschurqr} we compared the naive Jacobi with the triangular preconditioners, Schur and QR with $n$ iterations, in single and half precision. For the Schur preconditioner, the results are slightly better than for MATLAB's \texttt{eig}, while the difference in favor of the QR preconditioner is significant.

\begin{figure}[ht]
\includegraphics[width=.45\linewidth]{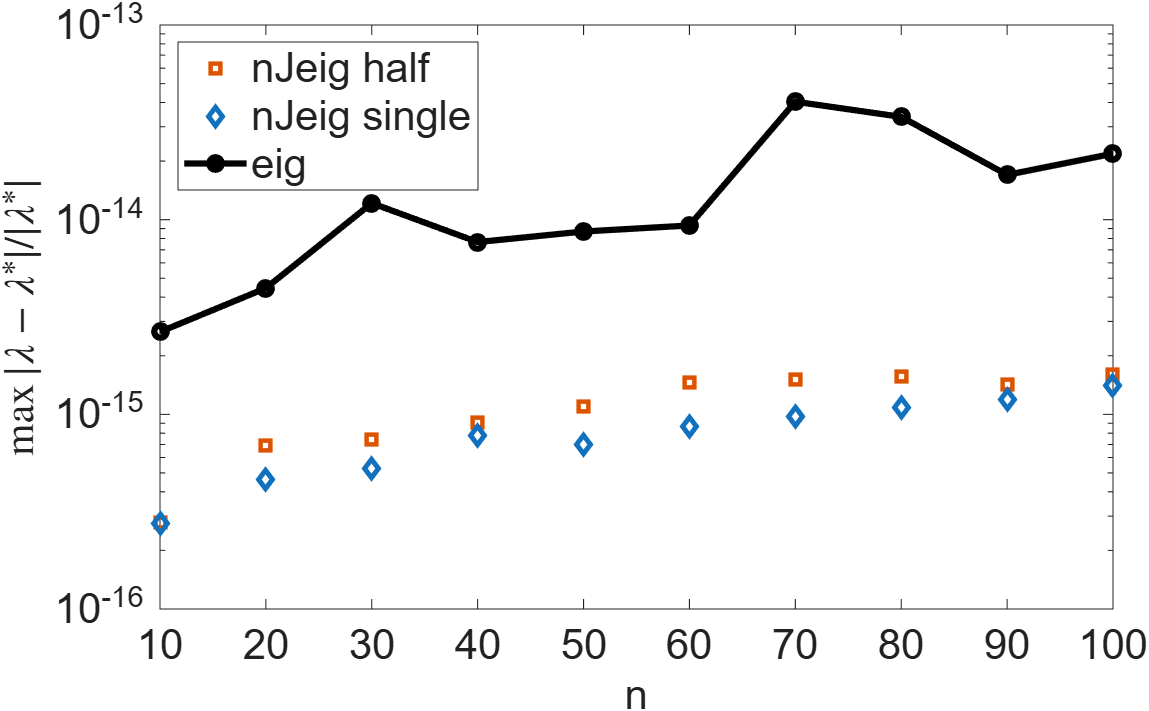} \quad \includegraphics[width=.45\linewidth]{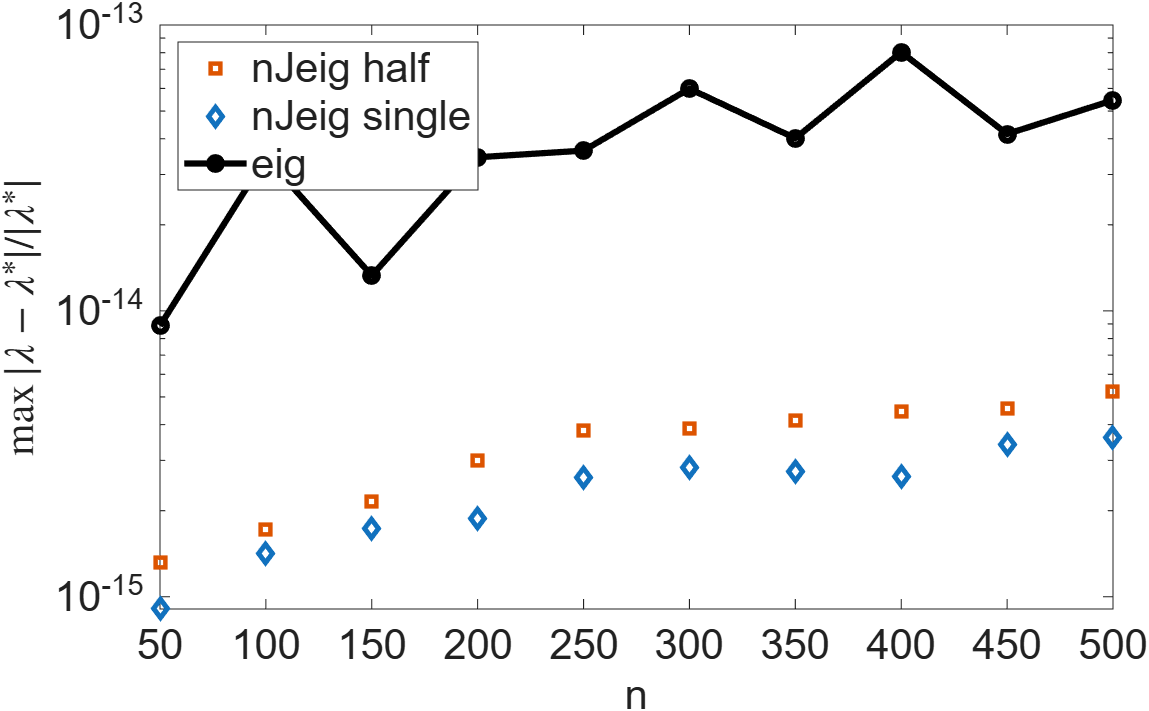}
\caption{Accuracy of the mixed-precision naive Jacobi with the \texttt{eig} preconditioner}\label{fig:mp_acceig}
\end{figure}

\begin{figure}[ht]
\includegraphics[width=.45\linewidth]{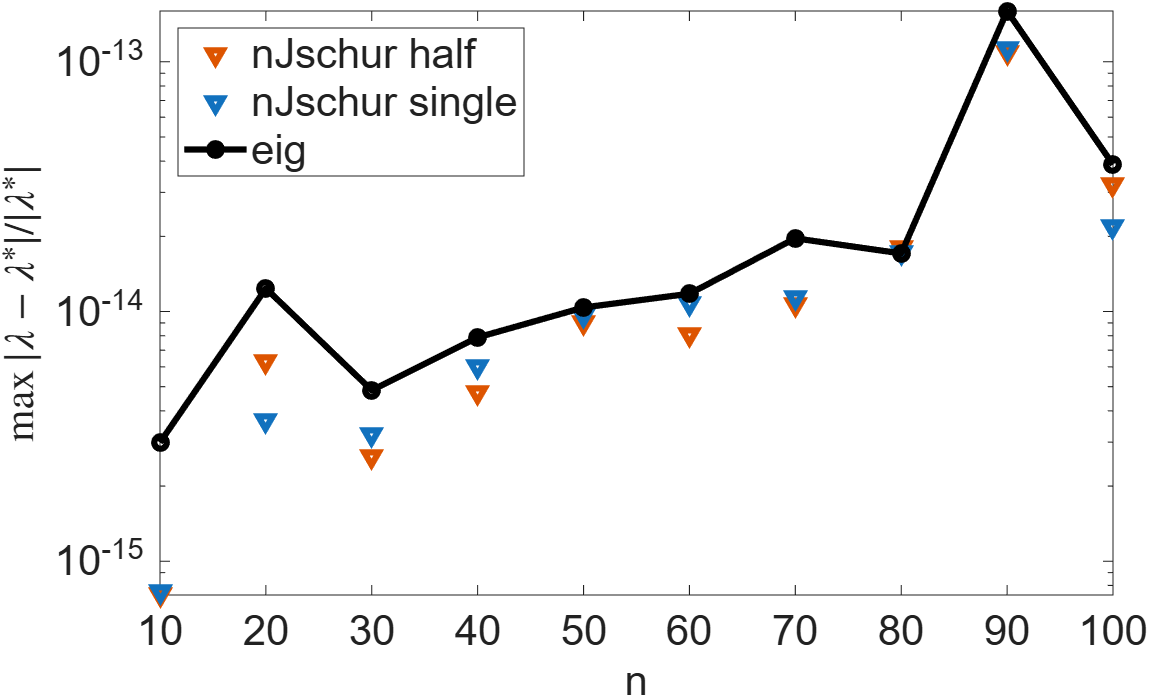} \quad \includegraphics[width=.45\linewidth]{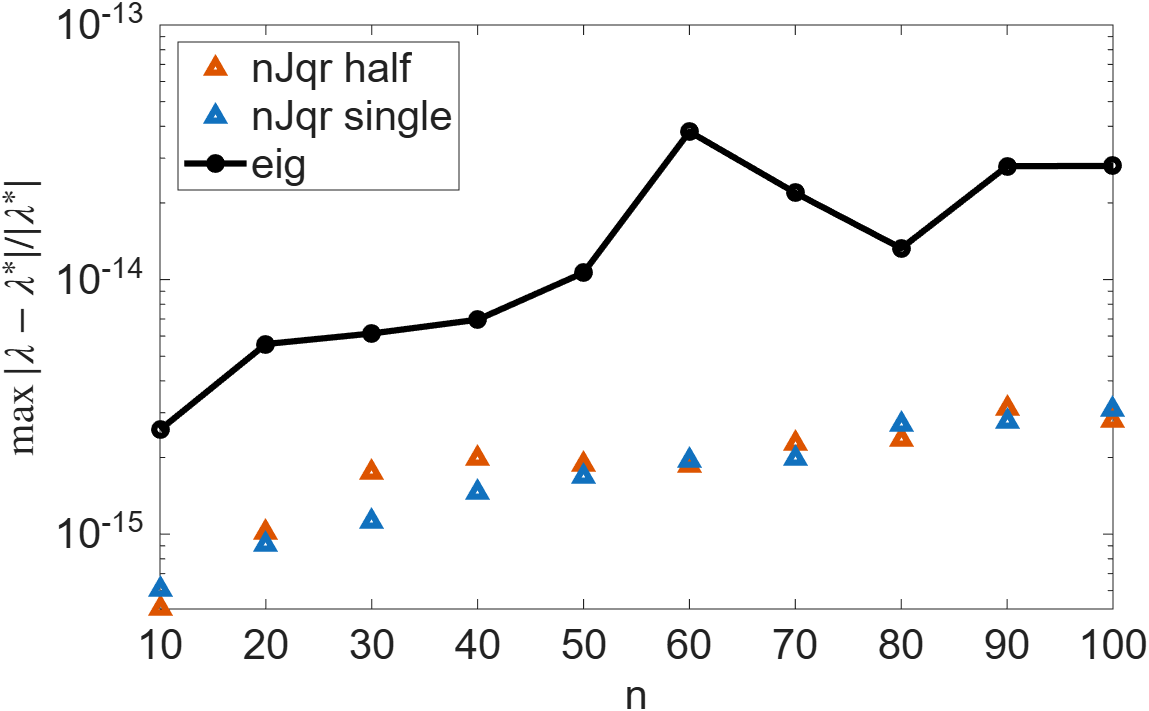}
\caption{Accuracy of the mixed-precision naive Jacobi with the Schur (left) and QR (right) preconditioner}\label{fig:mp_accschurqr}
\end{figure}

\subsubsection{Varying matrix condition number}

\begin{figure}[ht]
\includegraphics[width=.45\linewidth]{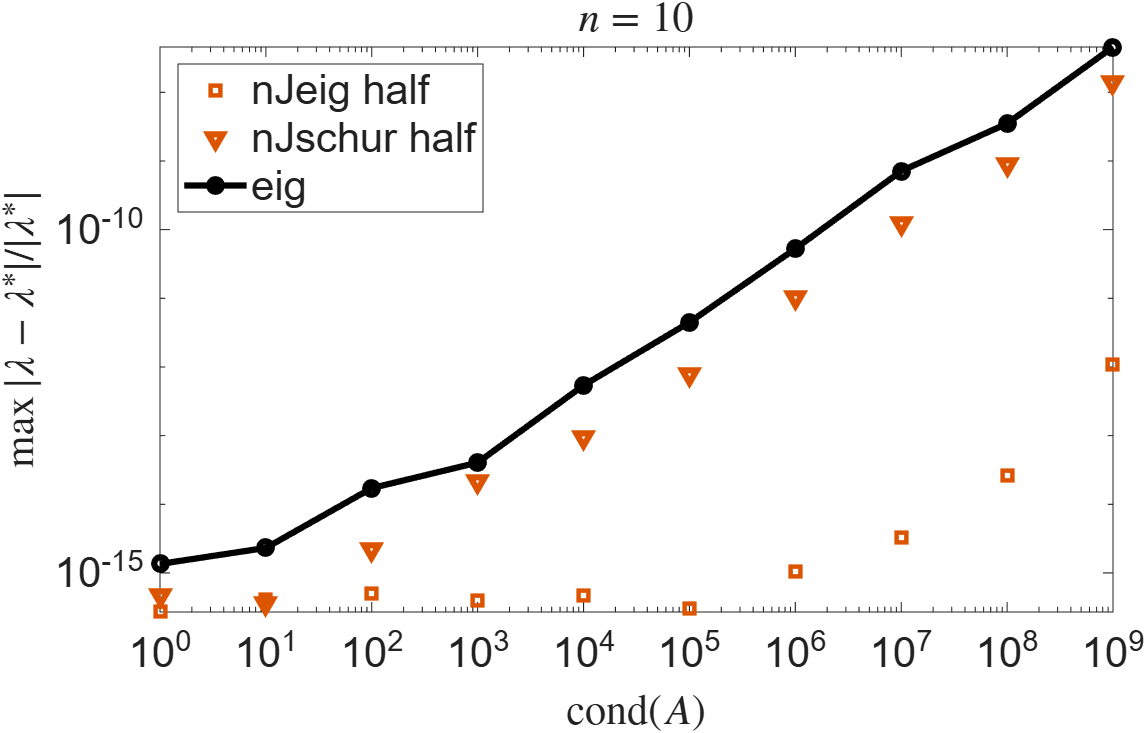} \quad \includegraphics[width=.45\linewidth]{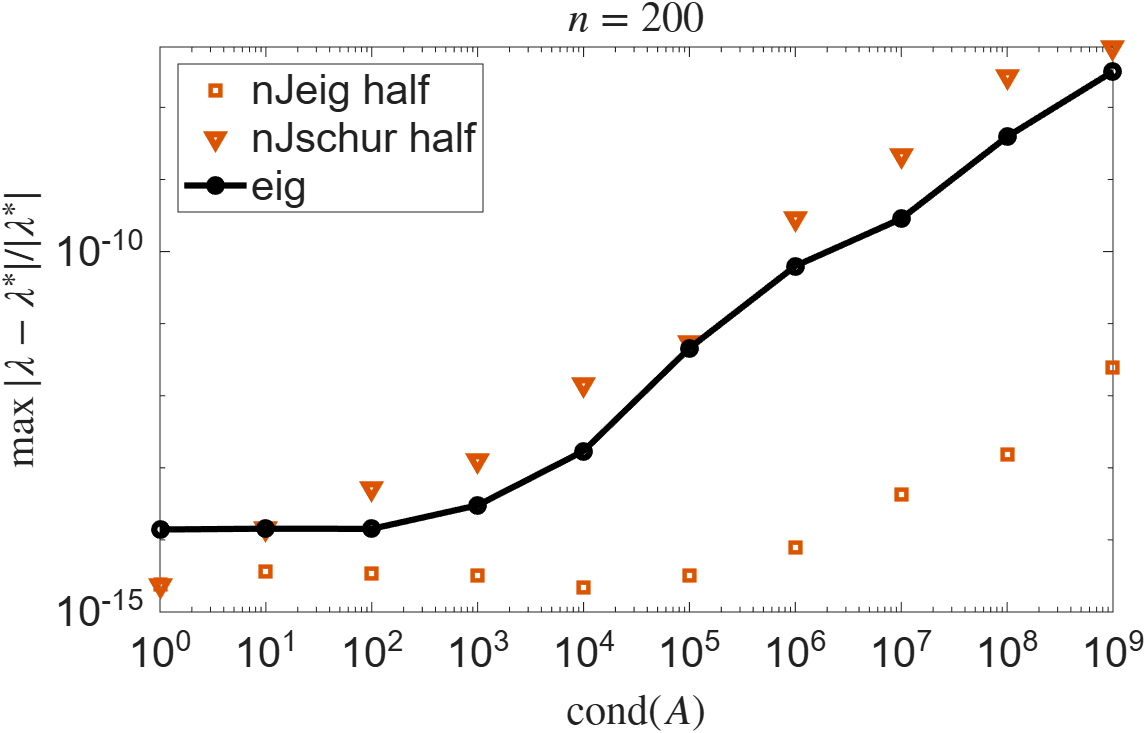}
\caption{Accuracy of the mixed-precision naive Jacobi with the \texttt{eig} and Schur preconditioner}\label{fig:mp_cond}
\end{figure}

For the fixed $n$ we tested preconditioned naive Jacobi algorithm on ill-conditioned matrices. We applied the \texttt{eig} and Schur preconditioners, both in half precision. Tested matrices with the condition number $c$ are formed in the following way:
\begin{verbatim}
    rng(1);
    s=[1, 1/c+(1-1/c)*rand(1,n-2), 1/c];
    [U,~]=qr(complex(rand(n),rand(n)));
    [V,~]=qr(complex(rand(n),rand(n)));
    A=V*diag(s)*U';
\end{verbatim}
We used $c$ between $1$ and $10^9$, for $n=10$ and $n=200$. In Figure~\ref{fig:mp_cond} we can observe that the naive Jacobi preconditioned by \texttt{eig} is much better than MATLAB's \texttt{eig}, while the Schur preconditioner provided very decent results, comparable to MATLAB's \texttt{eig}.

\subsection{Special test matrices}

We tested our algorithm on two special types of matrices. The first one is a tridiagonal matrix formed as
\begin{verbatim}
    rng(1);
    d=rand(n,1)+rand(n,1)*1i;
    c=ones(n-1,1);
    e=ones(n-1,1);
    A=full(gallery('tridiag',c,d,e));
\end{verbatim}
In Figure~\ref{fig:tridiag} we present the relative accuracy and convergence results for the preconditioned algorithm, using \texttt{eig} and \texttt{schur} preconditioners, both in half and single precision. The preconditioned naive Jacobi always gave more accurate results. On the $300\times300$ matrix, the algorithm converged in between two and five cycles, depending on the preconditioning strategy.

\begin{figure}[ht]
\includegraphics[width=.45\linewidth]{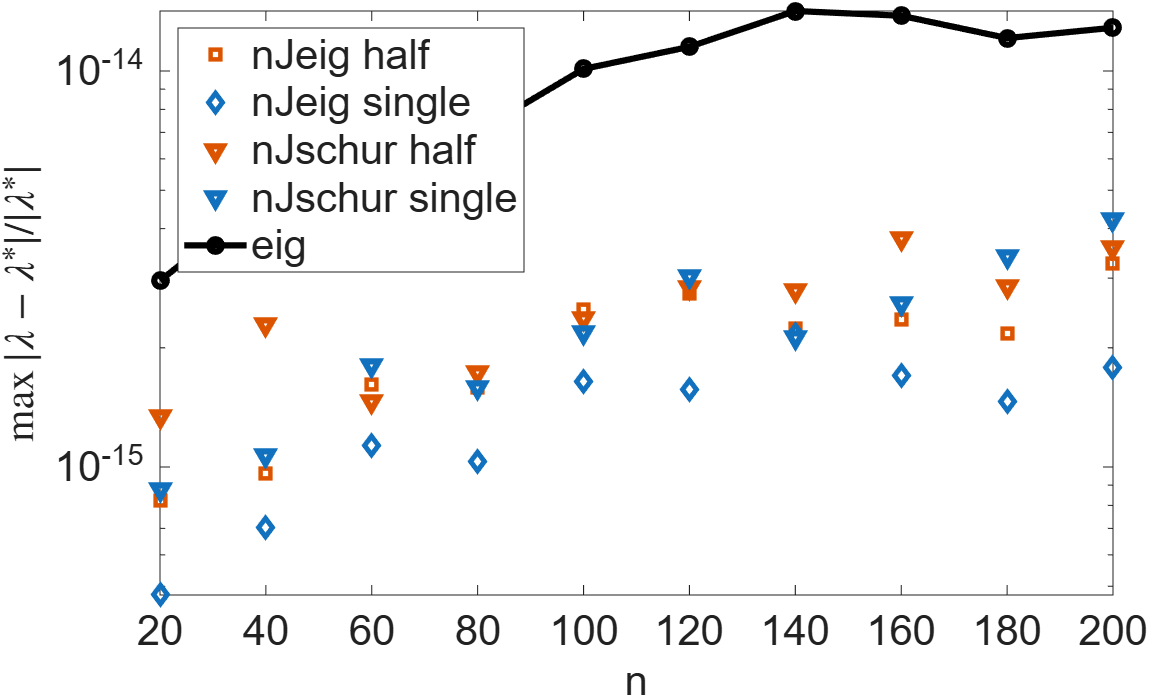} \quad \includegraphics[width=.45\linewidth]{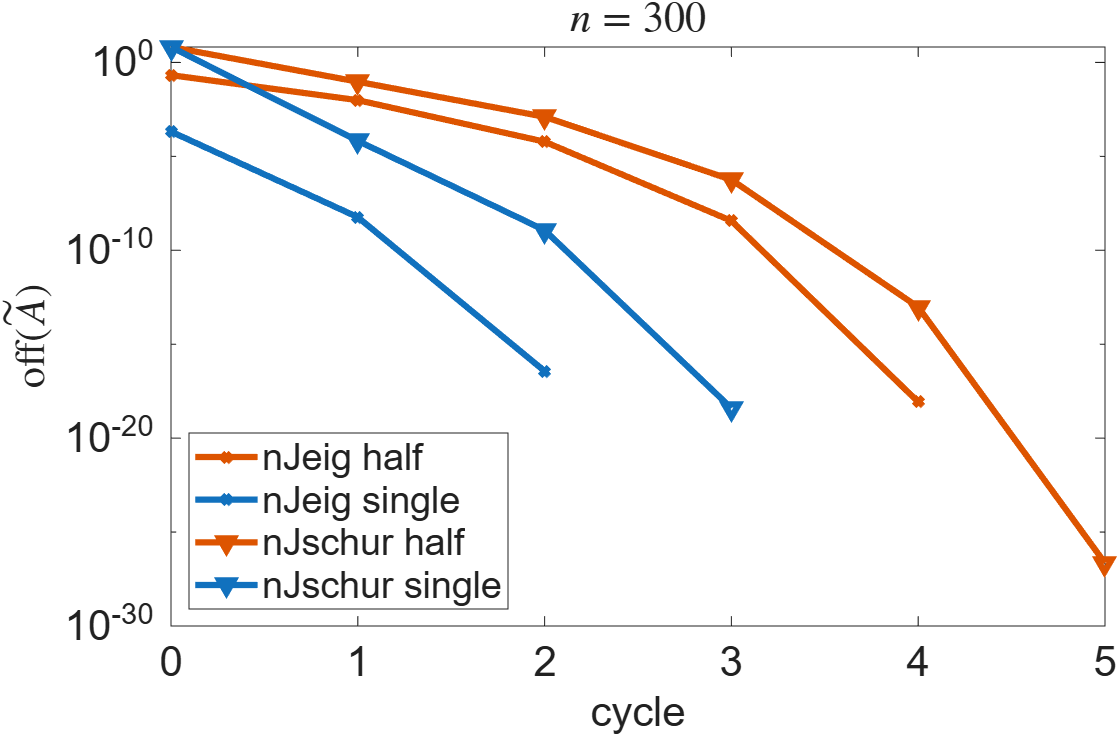}
\caption{Accuracy and convergence of the mixed-precision naive Jacobi on the tridiagonal matrix}\label{fig:tridiag}
\end{figure}

\begin{figure}[ht]
\includegraphics[width=.45\linewidth]{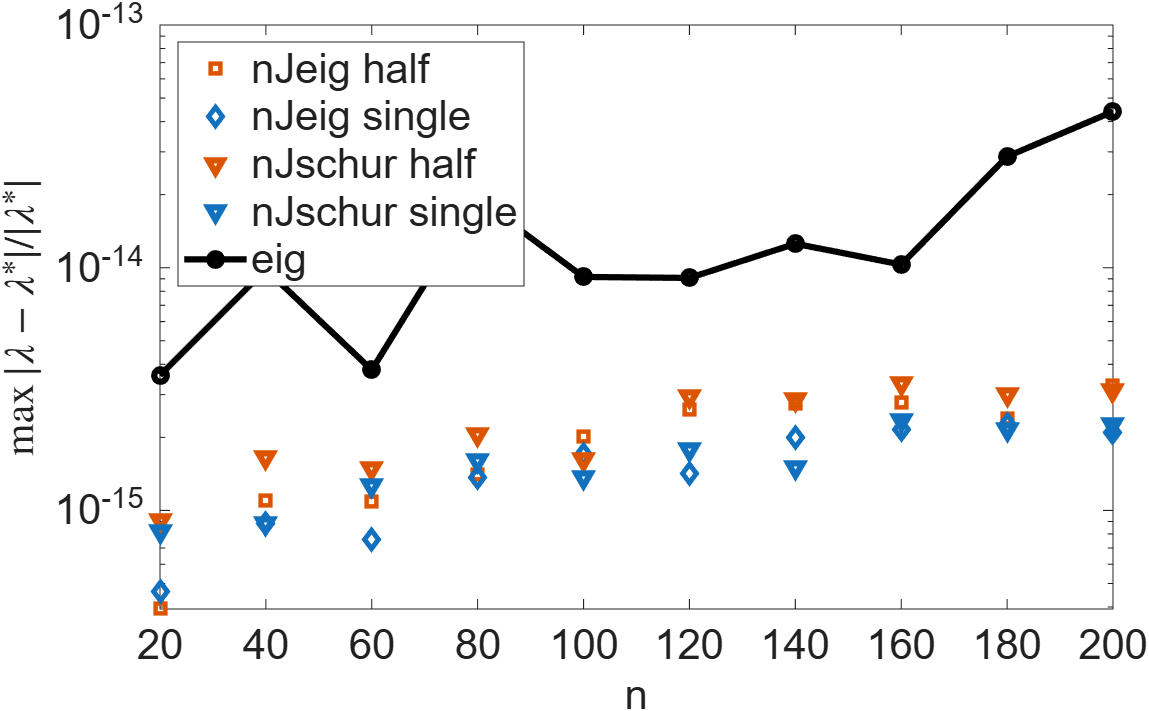} \quad \includegraphics[width=.45\linewidth]{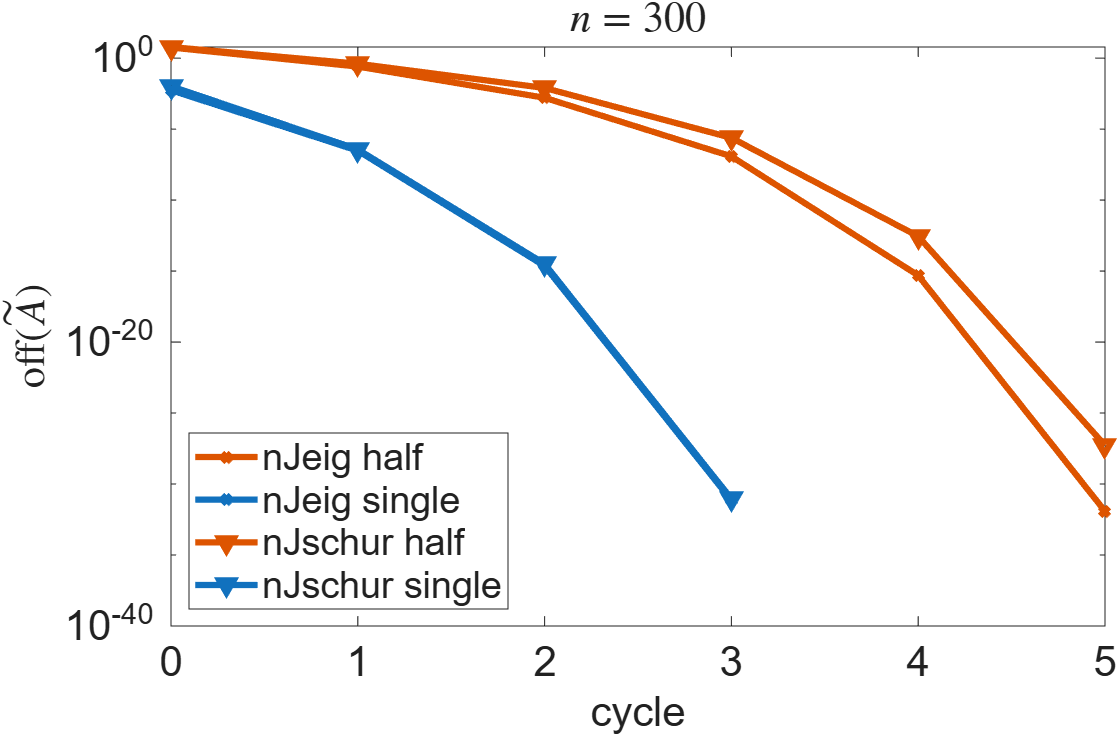}
\caption{Accuracy and convergence of the mixed-precision naive Jacobi on the perturbed ``hanowa'' matrix}\label{fig:hanowa}
\end{figure}

Then, we tested the ``hanowa'' matrix from the MATLAB gallery. That is a real matrix of the form
$$\begin{bmatrix} -I & -D \\ D & -I \end{bmatrix},$$
on which we added the complex perturbation,
\begin{verbatim}
    rng(1);
    A=gallery('hanowa',n)+0.0001*rand(n,n)*1i;
\end{verbatim}
The results are given in Figure~\ref{fig:hanowa} and they are very similar to those for the triangular matrix.

\section{Conclusion}\label{sec:conclusion}

We have revisited the naive Jacobi algorithm for the eigenvalue problem of general complex matrices with simple eigenvalues. We established its asymptotic quadratic convergence for matrices sufficiently close to diagonal form and derived an explicit, verifiable sufficient condition for this convergence. For the row-wise pivot strategy, we also showed that an upper-triangular matrix is diagonalized in one cycle and that matrices sufficiently close to upper-triangular form subsequently enter the quadratic convergence region.

To extend these local results to general input matrices, we introduced mixed-precision preconditioners based on a low-precision eigenvalue decomposition, Schur decomposition, or QR iteration. Numerical experiments show that the theoretical conditions are conservative and that the preconditioned method reliably converges well beyond the guaranteed region. The resulting algorithms attain excellent eigenvalue accuracy, up to one order of magnitude better than MATLAB's \texttt{eig} function, on the matrix families considered in our experiments. These results indicate that mixed-precision preconditioning makes the naive Jacobi algorithm a promising high-accuracy option for solving the eigenvalue problem.

\section*{Acknowledgments}
The authors thank Kre\v simir Veseli\'c for suggesting this research topic and for his valuable insights. The authors also thank Zlatko Drma\v{c} for useful discussion.

\section*{Declaration of AI Use}
OpenAI’s ChatGPT 5.4 Pro was used to sharpen the inequality~\eqref{eq:tmassumption} from Lemma~\ref{tm:convpom}.

\bibliographystyle{siam}
\bibliography{naiveJacobi.bib}

@article{BarlowDemmel1990,
    AUTHOR = {Barlow, Jesse and Demmel, James},
     TITLE = {Computing accurate eigensystems of scaled diagonally dominant
              matrices},
   JOURNAL = {SIAM J. Numer. Anal.},
  FJOURNAL = {SIAM Journal on Numerical Analysis},
    VOLUME = {27},
      YEAR = {1990},
    NUMBER = {3},
     PAGES = {762--791},
      ISSN = {0036-1429},
   MRCLASS = {65F15},
  MRNUMBER = {1041262},
MRREVIEWER = {Alan\ L.\ Andrew},
       DOI = {10.1137/0727045},
       URL = {https://doi-org.ezproxy.nsk.hr/10.1137/0727045},
}

@phdthesis{Zach,
    author ={Zacharias, Wolfgang} ,
    title = {{\"U}ber die {E}igenwertberechnung mittels primitiver {J}acobi-{\"A}hnlicher Verfahren},
    school = {Fernuniversit\"at Hagen},
    year = {1989},
    URL={https://www.fernuni-hagen.de/MATHPHYS/veselic/downloads/zacharias_w_diss_89.pdf}
}

@manual{advanpix,
  title        = {Multiprecision Computing Toolbox for {MATLAB}},
  author       = {{Advanpix}},
  organization = {Advanpix},
  address      = {Tokyo, Japan},
  year         = {2026},
  note         = {Version 5.2, \url{https://www.advanpix.com/}}
}

@article{Zhang2025,
    author = {Zhang, Zhiyuan and Bai, Zheng-Jian},
    title = {A mixed precision preconditioned {J}acobi method for the symmetric eigenvalue problem},
    journal = {arXiv:2211.03339v2 [math.NA]},
    url={https://arxiv.org/abs/2211.03339v2},
    year = {2025}
}

@conference{VeselicNJ,
    author = {Veseli\'c, K.},
    booktitle = {IV {I}nternational Workshop on Accurate Solution of Eigenvalue Problems, {S}plit, {C}roatia},
    title = {Naive {J}acobi Algorithms for General Matrices} ,
    year = {June 24--27, 2002}
}

@book{LAPACK,
author={Anderson, E. and Bai,Z. and  Bischof, C. and et al.},
title={{LAPACK} {U}ser's {G}uide},
edition={3rd ed.},
publisher={{SIAM}, {P}hiladelphia},
year={1999},
doi={10.1137/1.9780898719604}
}

@article{Higham2025,
    AUTHOR = {Higham, Nicholas J. and Tisseur, Fran\c coise and Webb, Marcus
              and Zhou, Zhengbo},
     TITLE = {Computing accurate eigenvalues using a mixed-precision
              {J}acobi algorithm},
   JOURNAL = {SIAM J. Matrix Anal. Appl.},
  FJOURNAL = {SIAM Journal on Matrix Analysis and Applications},
    VOLUME = {46},
      YEAR = {2025},
    NUMBER = {4},
     PAGES = {2423--2448},
      ISSN = {0895-4798,1095-7162},
   MRCLASS = {65F15 (15A18 65F08)},
  MRNUMBER = {4979579},
       DOI = {10.1137/25M1723748},
       URL = {https://doi-org.ezproxy.nsk.hr/10.1137/25M1723748},
}

@book{Golub,
    AUTHOR = {Golub, Gene H. and Van Loan, Charles F.},
     TITLE = {Matrix computations},
    SERIES = {Johns Hopkins Studies in the Mathematical Sciences},
   EDITION = {Fourth},
 PUBLISHER = {Johns Hopkins University Press, Baltimore, MD},
      YEAR = {2013},
     PAGES = {xiv+756},
      ISBN = {978-1-4214-0794-4; 1-4214-0794-9; 978-1-4214-0859-0},
   MRCLASS = {65-02 (65Fxx)},
  MRNUMBER = {3024913},
MRREVIEWER = {J\"org\ Liesen},
}

@standard{IEEE,
  author        = {{IEEE}},
  title        = {Standard for Floating-Point Arithmetic},
  organization = {Institute of Electrical and Electronics Engineers},
  number       = {IEEE Std 754-2019},
  year         = {2019},
  doi          = {10.1109/IEEESTD.2019.8766229}
}

@article {ShroffSchreiber89,
    AUTHOR = {Shroff, Gautam and Schreiber, Robert},
     TITLE = {On the convergence of the cyclic {J}acobi method for parallel
              block orderings},
   JOURNAL = {SIAM J. Matrix Anal. Appl.},
  FJOURNAL = {SIAM Journal on Matrix Analysis and Applications},
    VOLUME = {10},
      YEAR = {1989},
    NUMBER = {3},
     PAGES = {326--346},
      ISSN = {0895-4798},
   MRCLASS = {65F15 (65Y05)},
  MRNUMBER = {1003103},
MRREVIEWER = {Bo\ K\aa gstr\"om},
       DOI = {10.1137/0610025},
       URL = {https://doi-org.ezproxy.nsk.hr/10.1137/0610025},
}

@article {Hari15,
    AUTHOR = {Hari, Vjeran},
     TITLE = {Convergence to diagonal form of block {J}acobi-type methods},
   JOURNAL = {Numer. Math.},
  FJOURNAL = {Numerische Mathematik},
    VOLUME = {129},
      YEAR = {2015},
    NUMBER = {3},
     PAGES = {449--481},
      ISSN = {0029-599X,0945-3245},
   MRCLASS = {65F15},
  MRNUMBER = {3311458},
MRREVIEWER = {Meisam\ Sharify},
       DOI = {10.1007/s00211-014-0647-8},
       URL = {https://doi-org.ezproxy.nsk.hr/10.1007/s00211-014-0647-8},
}

@article {BKHari17,
    AUTHOR = {Hari, Vjeran and Begovi\'c Kova\v{c}, Erna},
     TITLE = {Convergence of the cyclic and quasi-cyclic block {J}acobi
              methods},
   JOURNAL = {Electron. Trans. Numer. Anal.},
  FJOURNAL = {Electronic Transactions on Numerical Analysis},
    VOLUME = {46},
      YEAR = {2017},
     PAGES = {107--147},
      ISSN = {1068-9613},
   MRCLASS = {65F15 (15B57)},
  MRNUMBER = {3652027},
MRREVIEWER = {Rafikul\ Alam},
}

@article {BKHari24,
    AUTHOR = {Begovi\'c Kova\v{c}, Erna and Hari, Vjeran},
     TITLE = {Convergence of the complex block {J}acobi methods under the
              generalized serial pivot strategies},
   JOURNAL = {Linear Algebra Appl.},
  FJOURNAL = {Linear Algebra and its Applications},
    VOLUME = {699},
      YEAR = {2024},
     PAGES = {421--458},
      ISSN = {0024-3795,1873-1856},
   MRCLASS = {65F15 (65F25)},
  MRNUMBER = {4777800},
MRREVIEWER = {Raffaella\ Pavani},
       DOI = {10.1016/j.laa.2024.07.012},
       URL = {https://doi-org.ezproxy.nsk.hr/10.1016/j.laa.2024.07.012},
}

@article {Drmac09,
    AUTHOR = {Drma\v{c}, Zlatko},
     TITLE = {A global convergence proof for cyclic {J}acobi methods with
              block rotations},
   JOURNAL = {SIAM J. Matrix Anal. Appl.},
  FJOURNAL = {SIAM Journal on Matrix Analysis and Applications},
    VOLUME = {31},
      YEAR = {2009},
    NUMBER = {3},
     PAGES = {1329--1350},
      ISSN = {0895-4798,1095-7162},
   MRCLASS = {65F15 (15A18)},
  MRNUMBER = {2587780},
MRREVIEWER = {Ross\ A.\ Lippert},
       DOI = {10.1137/090748548},
       URL = {https://doi-org.ezproxy.nsk.hr/10.1137/090748548},
}

@article {Masc95,
    AUTHOR = {Mascarenhas, Walter F.},
     TITLE = {On the convergence of the {J}acobi method for arbitrary
              orderings},
   JOURNAL = {SIAM J. Matrix Anal. Appl.},
  FJOURNAL = {SIAM Journal on Matrix Analysis and Applications},
    VOLUME = {16},
      YEAR = {1995},
    NUMBER = {4},
     PAGES = {1197--1209},
      ISSN = {0895-4798},
   MRCLASS = {65F15},
  MRNUMBER = {1351466},
       DOI = {10.1137/S0895479890179631},
       URL = {https://doi-org.ezproxy.nsk.hr/10.1137/S0895479890179631},
}

@article {DrVe07-1,
    AUTHOR = {Drma\v{c}, Zlatko and Veseli\'c, Kre\v simir},
     TITLE = {New fast and accurate {J}acobi {SVD} algorithm. {I}},
   JOURNAL = {SIAM J. Matrix Anal. Appl.},
  FJOURNAL = {SIAM Journal on Matrix Analysis and Applications},
    VOLUME = {29},
      YEAR = {2008},
    NUMBER = {4},
     PAGES = {1322--1342},
      ISSN = {0895-4798,1095-7162},
   MRCLASS = {65F20 (15A12 15A18 15A23)},
  MRNUMBER = {2369298},
MRREVIEWER = {Dario\ Fasino},
       DOI = {10.1137/050639193},
       URL = {https://doi-org.ezproxy.nsk.hr/10.1137/050639193},
}

@article {DrVe07-2,
    AUTHOR = {Drma\v{c}, Zlatko and Veseli\'c, Kre\v simir},
     TITLE = {New fast and accurate {J}acobi {SVD} algorithm. {II}},
   JOURNAL = {SIAM J. Matrix Anal. Appl.},
  FJOURNAL = {SIAM Journal on Matrix Analysis and Applications},
    VOLUME = {29},
      YEAR = {2008},
    NUMBER = {4},
     PAGES = {1343--1362},
      ISSN = {0895-4798,1095-7162},
   MRCLASS = {65F20 (15A12 15A18 15A23)},
  MRNUMBER = {2369299},
MRREVIEWER = {Dario\ Fasino},
       DOI = {10.1137/05063920X},
       URL = {https://doi-org.ezproxy.nsk.hr/10.1137/05063920X},
}

@article {DeVe92,
    AUTHOR = {Demmel, James and Veseli\'c, Kre\v{s}imir},
     TITLE = {Jacobi's method is more accurate than {$QR$}},
   JOURNAL = {SIAM J. Matrix Anal. Appl.},
  FJOURNAL = {SIAM Journal on Matrix Analysis and Applications},
    VOLUME = {13},
      YEAR = {1992},
    NUMBER = {4},
     PAGES = {1204--1245},
      ISSN = {0895-4798},
   MRCLASS = {65F15 (65G05)},
  MRNUMBER = {1182723},
       DOI = {10.1137/0613074},
       URL = {https://doi-org.ezproxy.nsk.hr/10.1137/0613074},
}

@article {Matejas09,
    AUTHOR = {Mateja\v{s}, Josip},
     TITLE = {Accuracy of the {J}acobi method on scaled diagonally dominant
              symmetric matrices},
   JOURNAL = {SIAM J. Matrix Anal. Appl.},
  FJOURNAL = {SIAM Journal on Matrix Analysis and Applications},
    VOLUME = {31},
      YEAR = {2009},
    NUMBER = {1},
     PAGES = {133--153},
      ISSN = {0895-4798,1095-7162},
   MRCLASS = {65F15 (15A18 65F35)},
  MRNUMBER = {2487054},
MRREVIEWER = {Fabio\ Di Benedetto},
       DOI = {10.1137/070685993},
       URL = {https://doi-org.ezproxy.nsk.hr/10.1137/070685993},
}

@article {Mathias95,
    AUTHOR = {Mathias, Roy},
     TITLE = {Accurate eigensystem computations by {J}acobi methods},
   JOURNAL = {SIAM J. Matrix Anal. Appl.},
  FJOURNAL = {SIAM Journal on Matrix Analysis and Applications},
    VOLUME = {16},
      YEAR = {1995},
    NUMBER = {3},
     PAGES = {977--1003},
      ISSN = {0895-4798},
   MRCLASS = {65F15 (65G05)},
  MRNUMBER = {1337657},
MRREVIEWER = {M.\ Znojil},
       DOI = {10.1137/S089547989324820X},
       URL = {https://doi-org.ezproxy.nsk.hr/10.1137/S089547989324820X},
}

@article{EbPa90,
title = {Efficient implementation of {J}acobi algorithms and {J}acobi sets on distributed memory architectures},
journal={J. Parallel Distrib. Comput.},
fjournal = {Journal of Parallel and Distributed Computing},
volume = {8},
number = {4},
pages = {358-366},
year = {1990},
issn = {0743-7315},
doi = {https://doi.org/10.1016/0743-7315(90)90134-B},
url = {https://www.sciencedirect.com/science/article/pii/074373159090134B},
author = {P.J. Eberlein and Haesun Park},
}

@article {BeOVa15,
AUTHOR = {Be\v{c}ka, Martin and Ok\v{s}a, Gabriel and Vajter\v{s}ic, Mari\'an},
TITLE = {New dynamic orderings for the parallel one-sided
              block-{J}acobi {SVD} algorithm},
JOURNAL = {Parallel Process. Lett.},
FJOURNAL = {Parallel Processing Letters},
VOLUME = {25},
YEAR = {2015},
NUMBER = {2},
PAGES = {1550003, 19},
ISSN = {0129-6264,1793-642X},
MRCLASS = {65F15 (65Y05)},
MRNUMBER = {3361206},
DOI = {10.1142/S0129626415500036},
URL = {https://doi-org.ezproxy.nsk.hr/10.1142/S0129626415500036},
}

@article{Veselic76,
    AUTHOR = {Veseli\'c, K.},
     TITLE = {A convergent {J}acobi method for solving the eigenproblem of
              arbitrary real matrices},
   JOURNAL = {Numer. Math.},
  FJOURNAL = {Numerische Mathematik},
    VOLUME = {25},
      YEAR = {1975/76},
    NUMBER = {2},
     PAGES = {179--184},
      ISSN = {0029-599X,0945-3245},
   MRCLASS = {65F15},
  MRNUMBER = {403189},
MRREVIEWER = {F.\ Szidarovszky},
       DOI = {10.1007/BF01462271},
       URL = {https://doi-org.ezproxy.nsk.hr/10.1007/BF01462271},
}

@article{Veselic79,
    AUTHOR = {Veseli\'c, K.},
     TITLE = {On a class of {J}acobi-like procedures for diagonalising
              arbitrary real matrices},
   JOURNAL = {Numer. Math.},
  FJOURNAL = {Numerische Mathematik},
    VOLUME = {33},
      YEAR = {1979},
    NUMBER = {2},
     PAGES = {157--172},
      ISSN = {0029-599X,0945-3245},
   MRCLASS = {65F15},
  MRNUMBER = {549446},
MRREVIEWER = {Ludwig\ Elsner},
       DOI = {10.1007/BF01399551},
       URL = {https://doi-org.ezproxy.nsk.hr/10.1007/BF01399551},
}

@article {VeselicWenzel79,
    AUTHOR = {Veseli\'c, K. and Wenzel, H. J.},
     TITLE = {A quadratically convergent {J}acobi-like method for real
              matrices with complex eigenvalues},
   JOURNAL = {Numer. Math.},
  FJOURNAL = {Numerische Mathematik},
    VOLUME = {33},
      YEAR = {1979},
    NUMBER = {4},
     PAGES = {425--435},
      ISSN = {0029-599X,0945-3245},
   MRCLASS = {65F15},
  MRNUMBER = {553351},
       DOI = {10.1007/BF01399324},
       URL = {https://doi-org.ezproxy.nsk.hr/10.1007/BF01399324},
}

@article{BeOVa02,
AUTHOR = {Be\v{c}ka, Martin and Ok\v{s}a, Gabriel and Vajter\v{s}ic, Mari\'an},
title = {Dynamic ordering for a parallel block-{J}acobi {SVD} algorithm},
year = {2002},
issue_date = {February 2002},
publisher = {Elsevier Science Publishers B. V.},
address = {NLD},
volume = {28},
number = {2},
issn = {0167-8191},
url = {https://doi.org/10.1016/S0167-8191(01)00138-7},
doi = {10.1016/S0167-8191(01)00138-7},
journal={Parallel Comput.},
fjournal = {Parallel computing},
pages = {243–262},
}

@article {Singer12a,
    AUTHOR = {Singer, Sanja and Singer, Sa\v{s}a and Novakovi\'c, Vedran and
              Davidovi\'c, Davor and Bokuli\'c, Kre\v{s}imir and U\v{s}\'cumli\'c, Aleksandar},
     TITLE = {Three-level parallel {$J$}-{J}acobi algorithms for {H}ermitian
              matrices},
   JOURNAL = {Appl. Math. Comput.},
  FJOURNAL = {Applied Mathematics and Computation},
    VOLUME = {218},
      YEAR = {2012},
    NUMBER = {9},
     PAGES = {5704--5725},
      ISSN = {0096-3003,1873-5649},
   MRCLASS = {65F15 (65Y05)},
  MRNUMBER = {2870087},
       DOI = {10.1016/j.amc.2011.11.067},
       URL = {https://doi-org.ezproxy.nsk.hr/10.1016/j.amc.2011.11.067},
}

@article {Singer12b,
    AUTHOR = {Singer, Sanja and Singer, Sa\v{s}a and Novakovi\'c, Vedran and
              U\v{s}\'cumli\'c, Aleksandar and Dunjko, Vedran},
     TITLE = {Novel modifications of parallel {J}acobi algorithms},
   JOURNAL = {Numer. Algorithms},
  FJOURNAL = {Numerical Algorithms},
    VOLUME = {59},
      YEAR = {2012},
    NUMBER = {1},
     PAGES = {1--27},
      ISSN = {1017-1398,1572-9265},
   MRCLASS = {65F15 (65Y05)},
  MRNUMBER = {2892442},
MRREVIEWER = {Raffaella\ Pavani},
       DOI = {10.1007/s11075-011-9473-6},
       URL = {https://doi-org.ezproxy.nsk.hr/10.1007/s11075-011-9473-6},
}

@article {Mehl08,
    AUTHOR = {Mehl, Christian},
     TITLE = {On asymptotic convergence of nonsymmetric {J}acobi algorithms},
   JOURNAL = {SIAM J. Matrix Anal. Appl.},
  FJOURNAL = {SIAM Journal on Matrix Analysis and Applications},
    VOLUME = {30},
      YEAR = {2008},
    NUMBER = {1},
     PAGES = {291--311},
      ISSN = {0895-4798,1095-7162},
   MRCLASS = {65F15},
  MRNUMBER = {2399581},
MRREVIEWER = {Elias\ Jarlebring},
       DOI = {10.1137/060663246},
       URL = {https://doi-org.ezproxy.nsk.hr/10.1137/060663246},
}

@article {CaHi18,
    AUTHOR = {Carson, Erin and Higham, Nicholas J.},
     TITLE = {Accelerating the solution of linear systems by iterative
              refinement in three precisions},
   JOURNAL = {SIAM J. Sci. Comput.},
  FJOURNAL = {SIAM Journal on Scientific Computing},
    VOLUME = {40},
      YEAR = {2018},
    NUMBER = {2},
     PAGES = {A817--A847},
      ISSN = {1064-8275,1095-7197},
   MRCLASS = {65F10 (65F05 65F35 65G50)},
  MRNUMBER = {3775138},
MRREVIEWER = {Dimitrios\ Christou},
       DOI = {10.1137/17M1140819},
       URL = {https://doi-org.ezproxy.nsk.hr/10.1137/17M1140819},
}

@article {OkCa22,
    AUTHOR = {Oktay, Eda and Carson, Erin},
     TITLE = {Multistage mixed precision iterative refinement},
   JOURNAL = {Numer. Linear Algebra Appl.},
  FJOURNAL = {Numerical Linear Algebra with Applications},
    VOLUME = {29},
      YEAR = {2022},
    NUMBER = {4},
     PAGES = {Paper No. e2434, 24},
      ISSN = {1070-5325,1099-1506},
   MRCLASS = {65F10},
  MRNUMBER = {4468908},
       DOI = {10.1002/nla.2434},
       URL = {https://doi-org.ezproxy.nsk.hr/10.1002/nla.2434},
}

@article {HiMa22,
    AUTHOR = {Higham, Nicholas J. and Mary, Theo},
     TITLE = {Mixed precision algorithms in numerical linear algebra},
   JOURNAL = {Acta Numer.},
  FJOURNAL = {Acta Numerica},
    VOLUME = {31},
      YEAR = {2022},
     PAGES = {347--414},
      ISSN = {0962-4929,1474-0508},
   MRCLASS = {65G50 (65Fxx)},
  MRNUMBER = {4436588},
MRREVIEWER = {Christos\ Kravvaritis},
       DOI = {10.1017/S0962492922000022},
       URL = {https://doi-org.ezproxy.nsk.hr/10.1017/S0962492922000022},
}

@article{Tisseur26,
      title={Computing accurate singular values using a mixed-precision one-sided {J}acobi algorithm}, 
      author={Zhengbo Zhou and Françoise Tisseur and Marcus Webb},
      year={2026},
    journal = {arXiv:2602.18134 [math.NA]},
      url={https://arxiv.org/abs/2602.18134}, 
}

@article {Eberlein62,
    AUTHOR = {Eberlein, P. J.},
     TITLE = {A {J}acobi-like method for the automatic computation of
              eigenvalues and eigenvectors of an arbitrary matrix},
   JOURNAL = {J. Soc. Indust. Appl. Math.},
  FJOURNAL = {Journal of the Society for Industrial and Applied Mathematics},
    VOLUME = {10},
      YEAR = {1962},
     PAGES = {74--88},
      ISSN = {0368-4245},
   MRCLASS = {65.40},
  MRNUMBER = {139264},
MRREVIEWER = {H.\ H.\ Goldstine},
}

@article {BKPe26,
    AUTHOR = {Begovi\'c{} Kova\v{c}, Erna and Perkovi\'c, Ana},
     TITLE = {On the block {E}berlein diagonalization method},
   JOURNAL = {Linear Algebra Appl.},
  FJOURNAL = {Linear Algebra and its Applications},
    VOLUME = {738},
      YEAR = {2026},
     PAGES = {1--24},
      ISSN = {0024-3795,1873-1856},
   MRCLASS = {65F15},
  MRNUMBER = {5039027},
       DOI = {10.1016/j.laa.2026.02.033},
       URL = {https://doi-org.ezproxy.nsk.hr/10.1016/j.laa.2026.02.033},
}

@article {BeHa21,
    AUTHOR = {Hari, Vjeran and Begovi\'c{} Kova\v{c}, Erna},
     TITLE = {On the convergence of complex {J}acobi methods},
   JOURNAL = {Linear Multilinear Algebra},
  FJOURNAL = {Linear and Multilinear Algebra},
    VOLUME = {69},
      YEAR = {2021},
    NUMBER = {3},
     PAGES = {489--514},
}

\end{document}